\documentclass[11pt, a4paper]{amsart}

\usepackage{amsfonts,amsmath,amssymb, amscd,fullpage}
\usepackage[all]{xy}

\newtheorem{theorem}{Theorem}[section]
\newtheorem{lemma}[theorem]{Lemma}
\newtheorem{definition}[theorem]{Definition}
\newtheorem{proposition}[theorem]{Proposition}
\newtheorem{corollary}[theorem]{Corollary}

\theoremstyle{definition}
\newtheorem{rem}[theorem]{Remark}
\newtheorem{rems}[theorem]{Remarks}

\newcommand\pf{\begin{proof}}
\newcommand\epf{\end{proof}}

\renewcommand\H{\mathrm{H}}
\newcommand\B{\mathrm{B}}
\newcommand\ab{\mathrm{ab}}
\newcommand\alg{\mathrm{alg}}
\newcommand\Irr{\mathrm{Irr}}
\newcommand\eps{\varepsilon}
\newcommand\SW{\mathrm{Sw}}

\newcommand\ZZ{\mathbb{Z}}

\newcommand\quot{/ \!\! /}

\newcommand\Alg{\operatorname{Alg}}
\newcommand\Coalg{\operatorname{Coalg}}
\newcommand\Hopf{\operatorname{Hopf}}

\DeclareMathOperator{\Hom}{Hom}
\DeclareMathOperator{\Ext}{Ext}
\DeclareMathOperator{\Reg}{Reg}
\DeclareMathOperator{\reg}{reg}

\DeclareMathOperator{\Ker}{Ker}
\DeclareMathOperator{\HKer}{HKer}
\DeclareMathOperator{\id}{id}
\DeclareMathOperator{\Coker}{Coker}

\renewcommand{\Im}{\image}
\DeclareMathOperator{\image}{Im}

\numberwithin{equation}{section}

\title{The lazy homology of a Hopf algebra}

\author{Julien Bichon}
\address{Julien Bichon:
Laboratoire de Math\'ematiques,
Universit\'e Blaise Pascal,
Complexe universitaire des C\'ezeaux,
63177~Aubi\`ere Cedex, France}
\email{Julien.Bichon@math.univ-bpclermont.fr}

\author{Christian Kassel}
\address{Christian Kassel: 
Universit\'e de Strasbourg,
Institut de Recherche Math\'e\-ma\-tique Avanc\'ee,
CNRS - Universit\'e Louis Pasteur,
7 rue Ren\'{e} Descartes, 67084 Strasbourg, France}
\email{kassel@math.u-strasbg.fr}

\begin{document}

\begin{abstract}
To any Hopf algebra~$H$ we associate two
commutative Hopf algebras~$\H^{\ell}_1(H)$ and $\H^{\ell}_2(H)$, 
which we call the lazy homology Hopf algebras of~$H$.
These Hopf algebras are built in such a way that
they are ``predual" to the lazy cohomology groups based on 
the so-called lazy cocycles.
We establish two universal coefficient theorems relating the lazy cohomology
to the lazy homology and
we compute the lazy homology of the Sweedler algebra.
\end{abstract}

\maketitle



\hspace{3cm}

\section*{Introduction}

One way to construct Galois objects over a Hopf algebra~$H$ is to twist the 
multiplication of~$H$ with the help of a~$2$-cocycle. 
The Galois objects obtained in this way are called cleft Galois objects.
If $H$ is cocommutative, then the $2$-cocycles form a group under the convolution product;
quotienting this group by an appropriate subgroup of coboundaries, one obtains
a cohomology group, which had been constructed by Sweedler in~\cite{sw}.
This cohomology group is in bijection with the set of isomorphism classes of
cleft Galois objects.
When $H$ is no longer cocommutative, then
the convolution product of $2$-cocycles is not necessarily a $2$-cocycle
and the classification of cleft Galois objects is no longer given by a cohomology group
(to get around this difficulty, Aljadeff and the second-named author recently
introduced the concept of a generic $2$-cocycle in~\cite{ak}).

Now there are $2$-cocycles that behave well with respect to the convolution
product, namely the so-called \emph{lazy} $2$-cocycles,
which are the cocycles that commute with the product of~$H$;
it was observed by Chen~\cite{ch} that 
the convolution product of lazy $2$-cocycles is again a lazy $2$-cocycle.
Quotienting the group of lazy $2$-cocycles by certain coboundaries,
one obtains a group~$\H_{\ell}^2(H)$, which was introduced by Schauenburg~\cite{sc3}. 
If $H$ is cocommutative, then all $2$-cocycles are lazy and $\H_{\ell}^2(H)$
coincides with Sweedler's cohomology group mentioned above.
In particular, if $H = k[G]$ is a group algebra,
then $\H_{\ell}^2(H)$ coincides with the cohomology group~$H^2(G,k^{\times})$
of the group~$G$ acting trivially on the group~$k^{\times}$ 
of nonzero elements of the ground field~$k$.
This makes it natural to consider~$\H_{\ell}^2(H)$
as an analogue of the Schur multiplier for arbitrary Hopf algebras.
The lazy cohomology group $\H_{\ell}^2(H)$ was systematically investigated 
by Carnovale and the first-named author in~\cite{bc}; see also \cite{cc, cp, psvo}.
Lazy cocycles have been used to compute the Brauer group of a Hopf algebra.
The group~$\H_{\ell}^2(H)$ also allows to equip the category of projective
representations of~$H$ with the structure of a crossed $\pi$-category in
the sense of Turaev~\cite{tu}.

In group cohomology, the group $H^2(G,k^{\times})$ is related to 
the homology groups $H_1(G)$ and~$H_2(G)$ of~$G$
\textit{via} the so-called \emph{universal coefficient theorem}, 
which can be formulated as an exact sequence of the form
$$0 \to \Ext^1(H_1(G),k^{\times}) \to H^2(G,k^{\times}) \to \Hom(H_2(G),k^{\times}) \to 0 \, .$$
It is natural to ask whether for an arbitrary Hopf algebra~$H$ there exist 
``lazy homology groups" $\H^{\ell}_1(H)$ and $\H^{\ell}_2(H)$ related to
the lazy cohomology~$\H_{\ell}^2(H)$ in a similar way 
and coinciding with Sweedler-type homology groups
for cocommutative Hopf algebras.

In this paper we give a positive answer to this question. 
To each Hopf algebra~$H$ we associate two
commutative Hopf algebras~$\H^{\ell}_1(H)$ and~$\H^{\ell}_2(H)$, 
which we call the \emph{lazy homology Hopf algebras} of~$H$.
These Hopf algebras are based on tensors satisfying conditions that are dual to the conditions
defining lazy cocycles. When $H$ is cocommutative, then 
$\H^{\ell}_1(H)$ and~$\H^{\ell}_2(H)$ are part 
of an infinite sequence of commutative cocommutative Hopf algebras
that appear as the homology of a simplicial object in the category of
commutative cocommutative Hopf algebras.

We establish two universal coefficient theorems relating the lazy homology
to the lazy cohomology.
The first one states that the group $\Alg(\H^{\ell}_1(H),R)$ of algebra morphisms
from~$\H^{\ell}_1(H)$ to a commutative algebra~$R$ is isomorphic to the 
lazy cohomology group~$\H_{\ell}^1(H,R)$ of lazy $1$-cocycles with coefficients in~$R$.
The second universal coefficient theorem can be expressed as an exact sequence 
of groups of the form
\begin{equation}\label{UCT20}
1 \longrightarrow \Ext^1(H,R) \overset{\delta_{\#}}{\longrightarrow} \H^2_{\ell}(H,R)
\overset{\kappa}{\longrightarrow} \Alg(\H_2^{\ell}(H),R) \, .
\end{equation}
The group~$\Ext^1(H,R)$ is a Hopf algebra analogue of the usual~$\Ext^1$-group.
When $R$ coincides with the ground field~$k$ and the latter is algebraically closed, then
the homomorphism~$\kappa$ in~\eqref{UCT20} is an isomorphism:
\begin{equation*}
\H^2_{\ell}(H)  = \H^2_{\ell}(H,k) \cong \Alg(\H_2^{\ell}(H),k) \, .
\end{equation*}
We finally compute the lazy homology of the four-dimensional Sweedler algebra~$H_4$;
not surprisingly, our computation agrees with the computation 
of the lazy cohomology of~$H_4$ performed in~\cite{bc}.

The paper is organized as follows. 
Section~\ref{notation} is devoted to various preliminaries.
We first recall what lazy cocycles are and what lazy cohomology is; 
here we consider cocycles with values not only in the ground field,
but in an arbitrary commutative algebra. 
We next define exact sequences of Hopf algebras
and investigate the exactness of the induced sequences after application of the functor~$\Alg(-,R)$.
We also recall Takeuchi's free commutative Hopf algebra generated by a coalgebra, of which
we make an extensive use in the sequel. 

In Section~\ref{Swsimplicial}, after recalling the construction of Sweedler's cohomology,
we attach a simplicial commutative Hopf algebra $F(\Gamma_*(H))$ to each Hopf algebra~$H$.
When $H$ is cocommutative, then $F(\Gamma_*(H))$ is a simplicial object
in the category of commutative cocommutative Hopf algebras.
Since the latter is abelian, we can take the corresponding homology~$\H^{\SW}_*(H)$, 
which turns out to be an infinite sequence of commutative cocommutative Hopf algebras.
We explicitly compute the low-degree differential in the chain complex associated to
the simplicial object~$F(\Gamma_*(H))$. The constructions in this section
will be a precious guide for the definition of the lazy homology in Sections~\ref{lazyhom1} 
and~\ref{lazyhom2}. 

In Section~\ref{predual} we perform certain basic tensor constructions 
that allow us to ``predualize" the conditions defining lazy cocycles;
these constructions will be central in the subsequent sections.

Section~\ref{lazyhom1} starts with the definition of the first lazy homology 
Hopf algebra~$\H^{\ell}_1(H)$
and the proof of the isomorphism 
$\H_{\ell}^1(H,R)  \cong \Alg(\H^{\ell}_1(H),R)$ mentioned above. 
We next give an alternative definition of~$\H^{\ell}_1(H)$, which allows us to show that
if $H$ is the Hopf algebra of functions on a finite group~$G$, then
$\H^{\ell}_1(H)$ is isomorphic to the Hopf algebra of functions on the center of~$G$.
We also interpret $\H^{\ell}_1(H)$ as a kind of homology group, from which we deduce
that it coincides with the Sweedler-type homology Hopf algebra~$\H^{\SW}_1(H)$
(introduced in Section~\ref{Swsimplicial})
when $H$ is cocommutative. 

In Section~\ref{cosemisimple} we consider the case when
$H$ is a cosemisimple Hopf algebra over an algebraically field
of characteristic zero, and we compute~$\H^{\ell}_1(H)$
in terms of the ``universal abelian grading group'' of the tensor category of $H$-comodules.
As an application of our techniques, we recover M\"uger's construction of the center
of a compact group from its representation category.

In Section~\ref{lazyhom2} we define the second lazy homology Hopf algebra~$\H^{\ell}_2(H)$; 
it coincides with the Sweedler-type homology Hopf algebra~$\H^{\SW}_2(H)$ when $H$ is cocommutative.
The main result of this section is the exact sequence~\eqref{UCT20}. 

Section~\ref{Swalgebra} is devoted to the computation of the lazy homology Hopf algebras
of the Sweedler algebra.

\section{Notation and Preliminaries}\label{notation}

Throughout the paper, we fix a field~$k$ over which all
our constructions are defined. 
In particular, all linear maps are supposed to be $k$-linear
and unadorned tensor products mean tensor products over~$k$.

All algebras that we consider are associative unital $k$-algebras.
The unit of an algebra~$A$ will be denoted by~$1_A$, or~$1$ if no confusion is possible.
All algebra morphisms are supposed to preserve the units.
We denote the set of algebra morphisms from $A$ to~$A'$ by $\Alg(A,A')$.

All coalgebras considered are coassociative counital $k$-coalgebras. 
We denote the coproduct of a coalgebra by~$\Delta$ and its counit by~$\eps$.
We shall also make use of a Heyneman-Sweedler-type notation 
for the image 
$$\Delta(x) = x_1 \otimes x_2$$
of an element~$x$ of a coalgebra~$C$
under the coproduct, and we write
$$\Delta^{(2)}(x) = x_1 \otimes x_2 \otimes x_3$$
for the iterated coproduct 
$\Delta^{(2)} = (\Delta \otimes \id_C) \circ \Delta = (\id_C \otimes \Delta) \circ \Delta$,
and so~on. 

If $C$ and $C'$ are coalgebras, then $\Coalg(C,C')$ denotes the 
set of coalgebra morphisms from $C$ to~$C'$.
Similarly, if $H$ and $H'$ are Hopf algebras, then $\Hopf(H,H')$ denotes the 
set of Hopf algebra morphisms from $H$ to~$H'$.
In general, we assume that the reader is familiar with coalgebras and Hopf algebras,
in particular with~\cite{Sw}.

\subsection{Convolution monoids}\label{convolution}

Let $C$ be a coalgebra with coproduct~$\Delta$ and let $R$ be an algebra with product~$\mu$.
Recall that the convolution product of  $f,g \in \Hom(C,R)$ is defined by
\begin{equation}\label{convol}
f* g = \mu \circ (f \otimes g) \circ \Delta \in \Hom(C,R) \, .
\end{equation}
The convolution product endows $\Hom(C,R)$ with a monoid structure whose unit
is $\eta \varepsilon$, where $\eta: k \to R$ is the unit of~$R$ and
$\varepsilon : C \to k$ is the counit of~$C$. 
The group of convolution invertible elements in $\Hom(C,R)$ is denoted by~$\Reg(C,R)$. 
It is well known that the group $\Reg(C,R)$ is abelian if $C$ is cocommutative
and $R$ is commutative.

Suppose in addition that $C$ is an $H$-module coalgebra and $R$ is an $H$-module algebra
for some Hopf algebra~$H$. Then the subset $\Hom_H(C,R)$ of $H$-linear maps
in~$\Hom(C,R)$ is a submonoid of the latter for the convolution product.
We denote by~$\Reg_H(C,R)$ the group of convolution invertible elements in $\Hom_H(C,R)$.
This is a subgroup of~$\Reg(C,R)$; it is \emph{abelian} if $C$ is cocommutative
and $R$ is commutative.

\subsection{Lazy cocycles and cohomology}\label{lazycohom}
Let us recall several notions and notation from
\cite{bc,ch,sc3}, with some slight modification. 
We fix a \emph{commutative} algebra~$R$.

Given a coalgebra~$C$, the subgroup of \emph{lazy elements} of $\Reg(C,R)$ is defined by
\begin{equation}\label{lazy1}
\Reg_{\ell}(C,R) = \{ \mu \in \Reg(C,R) \ | \ 
\mu(x_1)\otimes x_2 = \mu(x_2)\otimes x_1 \; \text{for all} \;  x \in C\} \, .
\end{equation}
The previous equalities take place in~$R\otimes C$.

Let $H$ be a Hopf algebra. The subgroup of elements
of $\mu \in \Reg(H,R)$ satisfying $\mu(1)=1$ is denoted by $\Reg^1(H,R)$,
and we set
\begin{equation*}
\Reg^1_{\ell}(H,R) = \Reg^1(H,R) \cap \Reg_{\ell}(H,R) \, .
\end{equation*}

\begin{definition}\label{lazycoh1}
The first lazy cohomology group of $H$ with coefficients in~$R$
is the group
\begin{equation*}
\H_{\ell}^1(H,R) = \Alg(H,R)\cap \Reg^1_{\ell}(H,R) \, .
\end{equation*}
\end{definition}
 
The set $\reg_{\ell}^2(H,R)$ of \emph{lazy elements} in $\Reg(H \otimes H,R)$ 
consists of all those $\sigma \in \Reg(H\otimes H,R)$ such that
\begin{equation}\label{lazy2}
\sigma(x_1 \otimes y_1)\otimes x_2y_2  
= \sigma(x_2 \otimes y_2) \otimes x_1y_1 \in R \otimes H
\end{equation}
for all $x,y \in H$.
The subgroup of normalized elements of $\Reg(H \otimes H,R)$,
i.e., satisfying the additional condition
$$\sigma(x \otimes 1)=\varepsilon(x) \, 1_R =\sigma(1 \otimes x)$$ 
for all $x \in H$, is denoted by~$\Reg^2(H,R)$.
We set
\begin{equation}
\Reg^2_{\ell}(H,R) = \Reg^2(H,R)\cap \reg_{\ell}^2(H,R) \, .
\end{equation}

A \emph{left 2-cocycle} of~$H$ with coefficients in~$R$
is an element $\sigma \in \Reg^2(H,R)$ such that 
\begin{equation}\label{2cocycle}
\sigma(x_{1} \otimes  y_{1}) \,  \sigma(x_{2}y_{2} \otimes z) =
\sigma(y_{1} \otimes z_{1}) \, \sigma(x  \otimes y_{2} z_{2}) 
\end{equation}
for all $x,y,z \in H$.
We denote by~$Z^2(H,R)$ the set of such left $2$-cocycles and we set
\begin{equation}\label{ZL2}
Z^2_{\ell}(H,R) = Z^2(H,R)\cap \Reg^2_{\ell}(H,R) \, .
\end{equation}
The set $Z^2_{\ell}(H,R)$ is called the set of \emph{lazy $2$-cocycles with coefficients in~$R$}.
Chen~\cite{ch} was the first to observe that this set is a group
under the convolution product. 

For $\mu \in \Reg^1(H,R)$, define $\partial(\mu) \in \Reg^2(H,R)$ for all $x,y \in H$ by
\begin{equation}\label{coboundary}
\partial(\mu)(x\otimes y) = \mu(x_1) \, \mu(y_1) \, \mu^{-1}(x_2 y_2) \, ,
\end{equation}
where $\mu^{-1}$ is the convolution inverse of~$\mu$.
This defines a map
$$\partial : \Reg^1(H,R) \to \Reg^2(H,R) \, .$$
The image $B^2_{\ell}(H,R)$ of~$\partial$
is a central subgroup of $Z^2_{\ell}(H,R)$.

\begin{definition}\label{lazycoh2}
The second lazy cohomology group of~$H$ with coefficients in~$R$
is the quotient-group
\begin{equation*}
\H_{\ell}^2(H,R) = Z^2_{\ell}(H,R)/B_{\ell}^2(H,R) \, .
\end{equation*}
\end{definition}

As was pointed out in~\cite{bc}, there is no reason why the second lazy cohomology group
should be abelian in general, although it turns out to be abelian in all known computations.

When $R$ is the ground field~$k$, then $\H_{\ell}^i(H,R)$ ($i=1,2$) coincides with the group
denoted by~$H^i_{\textrm{L}}(H)$ in~\cite[Def.~1.7]{bc}. The above definition of
the lazy cohomology with values in an arbitrary commutative algebra~$R$ 
(rather than in the ground field) was
already suggested in~\cite[App.~A]{bc}.

\subsection{Hopf kernels and Hopf quotients}\label{HKer}
 
Following~\cite[Def.~1.1.5]{ad}, we say that a Hopf algebra morphism $\pi : H \to B$
is \emph{normal} if 
\begin{equation}\label{normal}
\{ x \in H \ | \ \pi(x_1) \otimes x_2 = 1 \otimes x \}
=  \{ x \in H \ | \ x_1 \otimes \pi(x_2) = x \otimes 1 \} \, .
\end{equation}
If $H$ is cocommutative, then any Hopf algebra morphism $\pi : H \to B$ is normal.

When $\pi$ is normal, then by~\cite[Lemma~1.1.4]{ad}, 
both sides of~\eqref{normal} coincide with
\begin{equation}\label{Hopfkernel}
\HKer(\pi) = \{ x \in H \ | \ x_1 \otimes \pi(x_2) \otimes x_3 
= x_1 \otimes 1 \otimes x_2\} \, .
\end{equation}
We call $\HKer(\pi)$ the \emph{Hopf kernel} of $\pi : H \to B$.
By~\cite[Lemma~1.1.3]{ad}, $\HKer(\pi)$ is a Hopf subalgebra of~$H$.
Observe that the counit $\varepsilon : H \to k$ is always normal and that
$\HKer(\varepsilon) = H$.

Let us illustrate the concept of a Hopf kernel in two special cases.
If $u : G \to G'$ is a homomorphism of groups and $k[u]: k[G] \to k[G']$
the induced morphism of Hopf algebras, then 
\begin{equation}\label{HKerG}
\HKer(k[u]) = k[\Ker(u)] \, .
\end{equation}

When $H = \mathcal O(G)$ and $B= \mathcal O(N)$, 
where $N \subset G$ are algebraic groups
and $k$ is an algebraically closed field of characteristic zero, then
the Hopf kernel of the natural Hopf algebra morphism 
$\mathcal O(G) \to \mathcal O(N)$
is isomorphic to~$\mathcal  O(G/\langle N\rangle)$,
where $\langle N \rangle$ is the normal algebraic subgroup of~$G$ generated by~$N$.  

Let $A \subset H$ be a Hopf subalgebra, and let 
$A^+$ = $\Ker(\varepsilon: A \to k)$
be the augmentation ideal of~$A$.
When $A^+H \nolinebreak = \nolinebreak HA^+$, 
we define the \emph{Hopf quotient}~$H \quot A$ 
to be the quotient Hopf algebra~$H /A^+H$:
\begin{equation}\label{quotient}
H \quot A = H/ A^+H \, .
\end{equation}

If $G_0$ is a normal subgroup of a group~$G$, then $k[G_0]$ is a 
Hopf subalgebra of~$k[G]$ and $k[G_0]^+ k[G] = k[G] k[G_0]^+$.
Moreover, the Hopf quotient $k[G]\quot k[G_0]$ is 
isomorphic to the Hopf algebra of the quotient group~$G/G_0$:
\begin{equation}\label{quotientG}
k[G]\quot k[G_0] \cong k[G/G_0] \, .
\end{equation}

\subsection{Short exact sequences}\label{exact-seq}

For any Hopf algebra~$H$ and any commutative algebra~$R$,
the convolution product~\eqref{convol} preserves the subset
$\Alg(H,R)$ of~$\Hom(H,R)$
and turns it into a group, the inverse of each~$f\in \Alg(H,R)$
being $f\circ S$, where $S$ is the antipode of~$H$.
The group~$\Alg(H,R)$ is abelian if $H$ is cocommutative.
If $\varphi : H \to H'$ is a morphism of Hopf algebras,
then the map 
$$\varphi^*: \Alg(H',R) \to \Alg(H,R)$$
defined by $\varphi^*(f) = f \circ \varphi$, where $f \in \Alg(H',R)$,
is a homomorphism of groups. In this way, $\Alg(H,R)$
becomes a contravariant functor from the category of
Hopf algebras (resp.\ cocommutative Hopf algebras)
to the category of groups (resp.\ abelian groups).

We have the following obvious fact.

\begin{lemma}\label{inject}
If $\varphi : H \to H'$ is a surjective morphism of Hopf algebras,
then
the homomorphism $\varphi^*: \Alg(H',R) \to \Alg(H,R)$
is injective for any commutative algebra~$R$.
\end{lemma}

Let
\begin{equation}\label{exact}
k \longrightarrow A \overset{\iota}{\longrightarrow}  
H \overset{\pi}{\longrightarrow}  B \longrightarrow k
\end{equation}
be a sequence of Hopf algebra morphisms.
Following~\cite{ad}, we say that the sequence~\eqref{exact} is \emph{exact}
if $\iota$ is injective, $\pi$ is surjective, and $\Ker(\pi) = \iota(A)^+ H$.
(In~\cite{ad} there is an additional condition,
which we do not need here;
this condition is automatically satisfied if the Hopf algebra~$H$ is
commutative; see~\cite[Prop.~1.2.4]{ad}.)
It follows from the third condition that $\pi \circ \iota = \eps\eta$.
Note also that if $\iota(A)^+ H = H \iota(A)^+$,
then $B$ is isomorphic to the Hopf quotient~$H \quot \iota(A)$ 
defined by~\eqref{quotient}.

\begin{proposition}\label{Alg-exact}
Any exact sequence
$k \longrightarrow A \overset{\iota}{\longrightarrow}  
H \overset{\pi}{\longrightarrow}  B \longrightarrow k$ 
of Hopf algebras induces an exact sequence of groups
\begin{equation*}
1 \longrightarrow \Alg(B,R) \overset{\pi^*}{\longrightarrow}  
\Alg(H,R) \overset{\iota^*}{\longrightarrow}  \Alg(A,R)
\end{equation*}
for any commutative algebra~$R$.
\end{proposition}

It follows under the conditions of the previous proposition that
$\Alg(B,R)$ is a \emph{normal} subgroup of~$\Alg(H,R)$.

\pf
In view of Lemma~\ref{inject} it is enough to check that
$\Ker \iota^* = \Im \pi^*$.
Let $f\in \Alg(H,R)$ be such that $\iota^*(f) = \eps\eta$.
This means that $f(i(x)) = \eps(x)\, 1$ for all $x\in A$.
Therefore, $f$ vanishes on~$\iota(A)^+$, hence on~$\iota(A)^+H$.
Since by assumption, $\Ker(\pi) = \iota(A)^+H$,
the map $f$ vanishes on~$\Ker(\pi)$.
It follows that there is $\bar{f} \in \Hom(B,R)$
such that $f = \bar{f} \circ \pi$. It is easy to check
that $\bar{f}$ is an algebra morphism.
Hence, $f = \pi^*(\bar{f}) \in \Im \pi^*$.
\epf

There is no reason why the homomorphism~$\iota^*$ of Proposition~\ref{Alg-exact}
should be surjective in general.
Nevertheless surjectivity holds in the following case.

\begin{proposition}\label{Alg-surj}
If $k \longrightarrow A \overset{\iota}{\longrightarrow}  
H \overset{\pi}{\longrightarrow}  B \longrightarrow k$ 
is an exact sequence of commutative Hopf algebras 
and $k$ is algebraically closed, then the sequence
\begin{equation*}
1 \longrightarrow \Alg(B,k) \overset{\pi^*}{\longrightarrow}  
\Alg(H,k) \overset{\iota^*}{\longrightarrow}  \Alg(A,k) \longrightarrow 1
\end{equation*}
is exact.
\end{proposition}

\pf
In view of Proposition~\ref{Alg-exact},
it is enough to check the surjectivity of 
$\iota^*: \Alg(H,k) \to \Alg(A,k)$.
We sketch a proof of the surjectivity of~$\iota^*$
following~\cite[Chap.~15, Exercise~3\,(a)]{wa}.
We may assume that the Hopf algebras $A$ and $H$ are finitely generated.
Let $\chi \in \Alg(A,k)$ and $I = \Ker \chi$.
Since $H$ is faithfully flat over~$A$ by~\cite[Th.~3.1]{ta72} (see also~\cite[Sect.~14.1]{wa}),
we have $IH \neq H$ by~\cite[Sect.~13.2]{wa}.
Therefore there is a maximal ideal~$J$ of~$H$
such that $IH \subset J$. The field~$H/J$ is a finite extension of~$k$;
since $k$ is algebraically closed, we have~$H/J = k$.
Then the natural projection $H \to H/J = k$ is an element
of~$\Alg(H,k)$ extending~$\chi$.
\epf

\subsection{The free commutative Hopf algebra generated by a coalgebra}\label{Takeuchi}

For any coalgebra~$C$ we denote by~$F(C)$ the
\emph{free commutative Hopf algebra} 
generated by the coalgebra $C$, as constructed by Takeuchi in~\cite{ta}:
it is the commutative algebra presented by generators $t(x)$ and ${t}^{-1}(x)$, $x \in C$,
subject to the following relations ($x,y \in H$, $\lambda \in k$):
\begin{equation*}
t(\lambda x + y) = \lambda t(x) + t(y)\, ,
\quad
{t}^{-1}(\lambda x + y)= \lambda {t}^{-1}(x) + {t}^{-1}(y) \, , 
\end{equation*}
and
\begin{equation}\label{ttbar}
t(x_1) \, {t}^{-1}(x_2) = \, \varepsilon(x)\, 1 = {t}^{-1}(x_1) \, t(x_2) \, .
\end{equation}

The coproduct~$\Delta$, the counit~$\varepsilon$, and the antipode~$S$ of $F(C)$ are given 
for all $x \in C$~by
\begin{equation}\label{Hopf-on-Takeuchi}
\begin{aligned}
\Delta \bigl( t(x) \bigr) & = t(x_1) \otimes t(x_2) \, , \quad
\Delta \bigl( {t}^{-1}(x) \bigr) = {t}^{-1}(x_2) \otimes {t}^{-1}(x_1) \, , \\
\varepsilon \bigl( t(x) \bigr) & = \varepsilon \bigl( {t}^{-1}(x) \bigr)  = \varepsilon(x) \, , \\
S \bigl( t(x) \bigr) & = {t}^{-1}(x) \, , \quad
S \bigl( {t}^{-1}(x) \bigr) = t(x) \, ,
\end{aligned}
\end{equation}
and the canonical map $t : C \to F(C) , \, x \mapsto t(x)$, 
is a coalgebra morphism. 
By~\cite{ta}, the Hopf algebra $F(C)$ has the following universal property.
 
\begin{proposition}\label{F-univprop}
For any commutative Hopf algebra~$R$,
the map $f \mapsto f \circ t$ induces an isomorphism
\begin{equation*}
\Hopf(F(C),R) \overset{\cong}{\longrightarrow} \Coalg(C,R) \, .
\end{equation*}
\end{proposition}

Thus $C \mapsto F(C)$ defines
a functor from the category of coalgebras to the category of commutative Hopf algebras.
There is another universal property of $F(C)$ that will be important
for us in Section~\ref{predual} when we ``predualize" lazy cohomology.
 
\begin{proposition}\label{Alg-Reg}
For any commutative algebra~$R$,
the map $f \mapsto f \circ t$ induces an isomorphism
\begin{equation*}
\Alg(F(C),R) \overset{\cong}{\longrightarrow} \Reg(C,R) \, .
\end{equation*}
\end{proposition}
 
\begin{proof}
For $f \in \Alg(F(C),R)$, the map $f \circ t$ is convolution invertible, with convolution
inverse given by $(f\circ t)^{-1}(x) = f({t}^{-1}(x))$ for all $x\in F(C)$. 
Therefore the map $\Alg(F(C),R) \to \Reg(C,R)$ is well defined; it is
easy to check that it is a homomorphism.

Given $\varphi \in \Reg(C,R)$ with inverse~$\varphi^{-1}$, it is easy to check
that there is a unique algebra morphism $\tilde{\varphi} : F(C) \to R$
such that $\tilde{\varphi}(t(x))= \varphi(x)$ and 
$\tilde{\varphi}({t}^{-1}(x)) = \varphi^{-1}(x)$ for all $x \in C$. 
The assignment $\varphi \mapsto \tilde{\varphi}$ defines a map
$\Reg(C,R)  \to \Alg(F(C),R)$,
which is inverse to the previous one.  
\end{proof}
 
A basic example is the case of the group algebra~$H=k[G]$ of a group~$G$: 
the resulting commutative cocommutative Hopf algebra $F(k[G])$ is 
isomorphic to the group algebra~$k[\ZZ[G]]$, 
where $\ZZ[G]$ is the free (additive) abelian group with basis~$G$.
Note that $k[\ZZ[G]]$ is isomorphic to the Laurent polynomial algebra
$k[\, t(g)^{\pm 1} \, | \, g\in G\, ]$.

\section{Sweedler's Simplicial Coalgebra}\label{Swsimplicial}

In Section~\ref{lazycohom} we defined the lazy cohomology of a Hopf algebra in a direct way.
When the Hopf algebra is cocommutative, the lazy cohomology coincides with
the cohomology introduced by Sweedler in~\cite{sw}.
In this section we first recall the construction of Sweedler's cohomology;
it is based on a simplicial coalgebra, which exists for any Hopf algebra.
Next, starting from the same simplicial coalgebra, to any cocommutative Hopf algebra
we associate a chain complex which yields a Sweedler-type homology theory.
Sweedler's cohomology groups are dual in an appropriate sense to these homology groups.
The lazy homology defined in Sections~\ref{lazyhom1} and~\ref{lazyhom2}
will be obtained from a careful modification of this chain complex.

\subsection{Sweedler's construction}\label{simplicial}
Let $H$ be a Hopf algebra. 
The starting point is the following simplicial $k$-module~$C_*(H)$, 
which was defined by Sweedler in~\cite[\S~2]{sw}:
\begin{equation*}
\cdots \xymatrix{C_{n+1}
\ar@<4ex>[r] \ar@<3ex>[r] \ar@<2ex>[r]_\cdots
& \ar@<1ex>[l] \ar[l]C_n \ar@<4ex>[r] \ar@<3ex>[r] \ar@<2ex>[r]_\cdots& 
\cdots \ar@<1ex>[l] \ar[l]}
\xymatrix{
C_2 \ar@<3ex>[r] \ar@<2ex>[r] \ar@<1ex>[r] &
C_1 \ar@<1ex>[r] \ar[r]\ar@<1ex>[l] \ar[l]
& C_0  \ar@<1ex>[l]} \, .
\end{equation*} 
Here $C_n = H^{\otimes n+1}$, and the face operators
$d_i : H^{\otimes n+1} \to H^{\otimes n}$ and degeneracy operators
$s_i : H^{\otimes n+1} \to H^{\otimes n+2}$ ($i = 0, \ldots, n$) are given by
\begin{equation*}
d_i(x^0 \otimes \cdots \otimes x^n) = 
\begin{cases}
 x^0 \otimes \cdots \otimes x^ix^{i+1} \otimes \cdots \otimes x^n 
 & \text{for $i=0, \ldots, n-1$} \, ,\\
 x^0 \otimes \cdots \otimes x^{n-1} \, \varepsilon(x^n) & \text{for $i=n$} \, ,
\end{cases}
\end{equation*}
and
\begin{equation*}
s_i(x^0 \otimes \cdots \otimes x^n) = x^0 \otimes \cdots \otimes x^i \otimes 1
\otimes \cdots \otimes x^n \quad \text{for $i=0, \ldots , n$} \, ,
\end{equation*}
for all $x^0, \ldots, x^n \in H$.
We give each~$C_n$ its natural $H$-module coalgebra structure,
where the coalgebra structure is induced from that of~$H$ and where the action of~$H$
is defined by its left multiplication on the leftmost $H$-tensorand in~$C_n$.
It is a direct verification that the face and degeneracy operators $d_i$, $s_i$ are 
$H$-module coalgebra morphisms.
It follows that $C_*(H)$ is a \emph{simplicial $H$-module coalgebra}. 
This holds for any Hopf algebra~$H$ without any additional assumption, such as cocommutativity.

Let $R$ be an $H$-module algebra. 
Applying the contravariant functor $\Reg_H(-,R)$,
(defined in Section~\ref{convolution}) to
the simplicial $H$-module coalgebra~$C_*(H)$,
we obtain the \emph{cosimplicial group} $\Reg_H(C_*(H),R)$.

If in addition $H$ is cocommutative, then
$C_*(H)$ is a simplicial object in the category of \emph{cocommutative}
$H$-module coalgebras. Hence,
$\Reg_H(C_*(H),R)$
is a \emph{cosimplicial abelian group}
for any \emph{commutative} $H$-module algebra~$R$.
Sweedler~\cite[\S~2]{sw} defined the \emph{cohomology of~$H$ in~$R$} 
as the cohomology of this cosimplicial abelian group.
We henceforth denote Sweedler's cohomology by~$\H_{\SW}^*(H,R)$:
\begin{equation}
\H_{\SW}^*(H,R) = H^* \bigl(\Reg_H(C_*(H),R)\bigr) \, .
\end{equation}
We stress that the groups~$\H_{\SW}^*(H,R)$ are defined only if 
$H$ is \emph{cocommutative} and $R$ is \emph{commutative}.

Suppose that $G$ is a group and that $H=k[G]$ is the corresponding group algebra
with its standard cocommutative Hopf algebra structure. 
Let $R$ be a commutative $H$-module algebra.
The elements of~$G$ act as automorphisms of~$R$ and hence act on the
subgroup~$R^\times$ of invertible elements of~$R$.
It follows from~\cite[\S~3]{sw} that 
$$\H_{\SW}^*(k[G],R) \cong H^*(G,R^\times)\, ,$$ 
where $H^*(G,R^\times)$ is the group
cohomology of~$G$ with coefficients in~$R^\times$.

\subsection{A Sweedler-type homology theory}\label{Sw-homology}

Let us consider again an arbitrary Hopf algebra~$H$.
Tensoring the simplicial $H$-module
coalgebra $C_*(H)$ of Section~\ref{simplicial}
degreewise by~$k \otimes_H-$, 
we obtain a new \emph{simplicial coalgebra} $\Gamma_*(H)$:
\begin{equation*}
\cdots \xymatrix{\Gamma_{n+1}
\ar@<4ex>[r] \ar@<3ex>[r] \ar@<2ex>[r]_\cdots
& \ar@<1ex>[l] \ar[l]\Gamma_n \ar@<4ex>[r] \ar@<3ex>[r] \ar@<2ex>[r]_\cdots& 
\cdots \ar@<1ex>[l] \ar[l]}
\xymatrix{
\Gamma_2 \ar@<3ex>[r] \ar@<2ex>[r] \ar@<1ex>[r] &
\Gamma_1 \ar@<1ex>[r] \ar[r]\ar@<1ex>[l] \ar[l]
& \Gamma_0 \ar@<1ex>[l] } \, ,
\end{equation*} 
where $\Gamma_n = H^{\otimes n}$ if $n>0$ and $\Gamma_0 = k$, 
and the face operators
$\partial_i : H^{\otimes n} \to H^{\otimes n-1}$ and degeneracy operators
$\sigma_i : H^{\otimes n} \to H^{\otimes n+1}$ are given by
\begin{equation}\label{face}
\partial_i(x^1 \otimes \cdots \otimes x^n) = 
\begin{cases}
\varepsilon({x^1}) \, x^2\otimes \cdots \otimes x^n & \text{for $i=0$} \, ,\\
 x^1 \otimes \cdots \otimes x^i x^{i+1} \otimes \cdots \otimes x^n & 
 \text{for $i=1, \ldots, n-1$} \, ,\\
 x^1 \otimes \cdots \otimes x^{n-1} \, \varepsilon(x^n) & \text{for $i=n$} \, ,
\end{cases}
\end{equation}
and 
\begin{equation*}
\sigma_i(x^1 \otimes \cdots \otimes x^n) =
\begin{cases}
1 \otimes x^1  \otimes \cdots \otimes x^n & \text{for $i=0$} \, , \\
x^1 \otimes \cdots \otimes x^i \otimes 1 \otimes \cdots \otimes x^n 
& \text{for $i=1, \ldots , n$} \, ,
\end{cases}
\end{equation*}
for all $x^1, \ldots, x^n \in H$.
The functor $F$ from coalgebras to commutative Hopf algebras
introduced in Section~\ref{Takeuchi}
now transforms the simplicial coalgebra~$\Gamma_*(H)$ into the 
\emph{simplicial commutative Hopf algebra} $F( \Gamma_*(H))$.

If in addition $H$ is \emph{cocommutative}, then $\Gamma_*(H)$
is a simplicial cocommutative coalgebra and therefore
$F(\Gamma_*(H))$ is a simplicial object in the category
of \emph{commutative cocommutative Hopf algebras}. 
By~\cite[Cor.~4.16]{ta72}, this category is abelian
(recall that $\Hopf(H,H')$ is an abelian group for the convolution 
if $H,H'$ are commutative cocommutative Hopf algebras)
and we can form kernels and quotients in it 
following~\eqref{Hopfkernel} and~\eqref{quotient}.
Therefore we can use the standard methods to form a chain complex
out of the simplicial object~$F( \Gamma_*(H))$. 
We pose the following definition.

\begin{definition}\label{HS-hom}
The Sweedler homology $\H^{\SW}_*(H)$ of a cocommutative Hopf algebra~$H$
is the homology of the chain complex 
\begin{equation}\label{complex}
\cdots \overset{}{\longrightarrow} F(H^{\otimes 3})
\overset{d^{\SW}_3}{\longrightarrow} F(H^{\otimes 2})
\overset{d^{\SW}_2}{\longrightarrow} F(H)
\overset{d^{\SW}_1}{\longrightarrow} F(k) 
\end{equation}
associated to the simplicial commutative cocommutative Hopf algebra $F( \Gamma_*(H))$:
\begin{equation*}
\H^{\SW}_*(H) = H_* \bigl( F\bigl( \Gamma_*(H) \bigr), d^{\SW}_*\bigr) \, .
\end{equation*}
\end{definition}

The homology groups $\H^{\SW}_*(H)$ are commutative cocommutative Hopf algebras.

Let us explicitly compute the low-degree differentials
$d^{\SW}_1$, $d^{\SW}_2$, $d^{\SW}_3$ in the complex~\eqref{complex}.

\begin{lemma}\label{d2-cocomm}
For $x,y,z \in H$, we have
$d^{\SW}_1 \bigl(t(x) \bigr) = \varepsilon(x)$ and
\begin{equation*}\label{d2b}
d^{\SW}_2 \bigl( t(x \otimes y) \bigr) = 
t(y_1) \, {t}^{-1}(x_1 y_2) \, t(x_2) \, ,
\end{equation*}
\begin{equation*}\label{d3b}
d^{\SW}_3 \bigl( t(x \otimes y \otimes z) \bigr) = 
t(y_1 \otimes z_1) \, {t}^{-1}(x_1 y_2 \otimes z_2) \, 
t(x_2 \otimes y_3 z_3) \, {t}^{-1}(x_3 \otimes y_4)\, .
\end{equation*}
\end{lemma}

Since $H$ is assumed to be cocommutative, 
$d^{\SW}_2$ and $d^{\SW}_3$ 
can be reformulated as follows:
\begin{equation}\label{d2a}
d^{\SW}_2 \bigl( t(x \otimes y) \bigr) = 
t(x_1) \, t(y_1) \, {t}^{-1}(x_2 y_2)
\end{equation}
and 
\begin{equation}\label{d3a}
d^{\SW}_3 \bigl( t(x \otimes y \otimes z) \bigr) = 
t(y_1\otimes z_1) \, t(x_1 \otimes y_2 z_2) \, 
{t}^{-1}(x_2y_3 \otimes z_3) \, {t}^{-1}(x_3 \otimes y_4) \, .
\end{equation}
Equations \eqref{d2a} and~\eqref{d3a} will be used in the proofs of
Propositions~\ref{HS1=HL1} and~\ref{HL2HS2}, respectively.

\begin{proof}
The map $d^{\SW}_1$ in $\Hopf(F(H), F(k))$ is defined by
$d^{\SW}_1 = \partial_0 * \partial_1^{-1}$. 
Therefore by~\eqref{face},
\begin{equation*}
d^{\SW}_1 \bigl(t(x) \bigr)
= \partial_0(x_1) \, \partial_1^{-1}(x_2)
= \varepsilon(x_1) \, \varepsilon(x_2) = \varepsilon(x) \, .
\end{equation*}

The map $d^{\SW}_2$ in $\Hopf(F(H^{\otimes 2}),F(H))$ is given by
\begin{equation*}
d^{\SW}_2 = \partial_0 * \partial_1^{-1} * \partial_2 \, .
\end{equation*}
By~\eqref{face},
\begin{align*}
d^{\SW}_2 \bigl( t(x \otimes y) \bigr)
& = t\bigl( \partial_0(x_1\otimes y_1) \bigr) \, 
{t}^{-1}\bigl( \partial_1(x_2 \otimes y_2) \bigr) \, 
t\bigl(\partial_2(x_3\otimes y_3) \bigr) \\
& = t \bigl(\varepsilon(x_1) \, y_1 \bigr) \, {t}^{-1}(x_2y_2) \, 
t \bigl(x_3 \, \varepsilon(y_3) \bigr) \\
& = t(y_1) \, {t}^{-1}\bigl(\varepsilon(x_1) \, x_2y_2 \, \varepsilon(y_3)\bigr) \, t(x_3) \\
& = t(y_1) \, {t}^{-1}(x_1 y_2) \, t(x_2) \, .
\end{align*}

Finally, the map $d^{\SW}_3$ in $\Hopf(F(H^{\otimes 3}),F(H^{\otimes 2}))$ is given by
\begin{equation*}
d^{\SW}_3 = \partial_0 * \partial_1^{-1} * \partial_2 * \partial_3^{-1}\, .
\end{equation*}
Again by~\eqref{face},
\begin{align*}
d^{\SW}_3 \bigl( t(x \otimes y \otimes z) \bigr)
& = t\bigl( \partial_0(x_1\otimes y_1\otimes z_1) \bigr) \, 
{t}^{-1}\bigl( \partial_1(x_2 \otimes y_2\otimes z_2) \bigr) \, \\
&  \hskip 45pt \times
t\bigl(\partial_2(x_3\otimes y_3\otimes z_3) \bigr) \,  
{t}^{-1}\bigl( \partial_3(x_4 \otimes y_4\otimes z_4) \bigr) \\
& = t \bigl(\varepsilon(x_1) \, y_1 \otimes z_1\bigr) \, {t}^{-1}(x_2y_2\otimes z_2) \\
&  \hskip 45pt \times
t(x_3\otimes y_3z_3) \, {t}^{-1} \bigl(x_4 \otimes y_4\, \varepsilon(z_4) \bigr) \\
& = t(y_1 \otimes z_1) \, {t}^{-1}(x_1 y_2 \otimes z_2) \, 
t(x_2 \otimes y_3 z_3) \, {t}^{-1}(x_3 \otimes y_4)\, .
\end{align*}
\end{proof}

The Sweedler homology of the Hopf algebra of a group is computed as follows.

\begin{proposition}\label{HSkG}
For any group~$G$ we have
\begin{equation*}
\H^{\SW}_*(k[G]) \cong k[H_*(G,\ZZ)]\, ,
\end{equation*}
where $H_*(G,\ZZ)$ is the homology of~$G$ with integral coefficients.
\end{proposition}

\begin{proof}
For $H=k[G]$ the simplicial abelian group~$C_*(H)$ is 
the one underlying the unnormalized bar $k$-resolution of~$G$.
Tensoring it by~$k \otimes_{k[G]}-$, we obtain the simplicial object $\Gamma_*(H)$,
whose associated chain complex homology has~$H_*(G,k)$ as homology.
Applying the functor $F$, we obtain a simplicial commutative Hopf algebra, 
whose component in degree~$n$ is the group algebra $k[\ZZ[G^n]]$.
We thus obtain the $k$-linearization of the simplicial complex whose homology is~$H_*(G,\ZZ)$.
We conclude using~\eqref{HKerG} and~\eqref{quotientG}.
\end{proof}

When $H$ is an arbitrary Hopf algebra, we do not know how to use directly the 
simplicial commutative Hopf algebra~$F( \Gamma_*(H))$ 
in order to build up something that looks like a chain complex.
In the subsequent sections we shall carefully modify the low-degree terms 
of~$F(\Gamma_*(H))$ in a manner that will be suitable 
for defining what we shall call lazy homology. 
More precisely, in order to define lazy homology
we predualize the laziness conditions at the level 
of the coalgebras~$H^{\otimes q}$ ($q= 1,2$) by taking appropriate quotients
(see Section~\ref{predual}), apply the functor~$F$, 
and mimic the construction of a complex out of a simplicial object
(Sections~\ref{lazyhom1} and~\ref{lazyhom2}).

\begin{rem}
Let $H$ be a cocommutative Hopf algebra and $R$ a commutative algebra.
We turn~$R$ into an $H$-module algebra by assuming that
$H$ acts \emph{trivially} on~$R$ via the counit~$\eps$. It is easy to check that
under these hypotheses there is an isomorphism
$$\Reg_H(C_*(H),R) \cong \Reg(\Gamma_*(H),R)$$
of cosimplicial abelian groups. 
In view of Proposition~\ref{Alg-Reg}, this leads to an isomorphism 
$$\Reg_H(C_*(H),R) \cong \Alg\bigl( F(\Gamma_*(H)),R \bigr)$$
of cosimplicial abelian groups, 
relating the cochain complex whose cohomology is~$\H_{\SW}^*(H,R)$
and the chain complex~\eqref{complex} whose homology is~$\H^{\SW}_*(H)$.
From this it is easy to deduce natural homomorphisms
\begin{equation}\label{hocosw}
\H_{\SW}^i(H,R) \to \Alg(\H^{\SW}_i(H),R)
\end{equation}
for all $i\geq 0$.
When the ground field $k$ is algebraically closed and $R=k$, 
one can use Sections~\ref{HKer}--\ref{exact-seq} and~\cite{ta72}
to show that the homomorphisms~\eqref{hocosw} are isomorphisms:
$$\H_{\SW}^i(H,k) \cong \Alg(\H^{\SW}_i(H),k)$$
for all $i\geq 0$. When $H= k[G]$ is a group algebra,
these isomorphisms reduce to the well-known isomorphisms 
$$H^i(G,k^{\times}) \cong \Hom(H_i(G, \ZZ),k^{\times})$$
in group cohomology (the group~$k^{\times}$ is divisible under the condition above).
\end{rem}

\subsection{Addendum: Homology with coefficients}\label{coefficients}

The content of this section will not be needed in the sequel. 
The construction of the simplicial coalgebra~$\Gamma_*(H)$
of Section~\ref{Sw-homology} is a special case of the following
construction.

Start again from a Hopf algebra~$H$ and the associated
simplicial coalgebra $C_*(H)$ defined in Section~\ref{simplicial}.
Let $C$ be a right $H$-module coalgebra.
In each degree~$n$ let us consider the tensor product
$$C \otimes_H  C_n(H) = C \otimes_H H^{\otimes n+1}\, ;$$
it is isomorphic to
$$\Gamma_n(H,C) = C \otimes H^{\otimes n}$$
under the map 
$x^0 \otimes x^1 \otimes \cdots \otimes x^n \mapsto 
x^0 x^1 \otimes \cdots \otimes x^n$, where $x^0 \in C$
and $x^1, \ldots, x^n \in H$.
The face and degeneracy operators of~$C_*(H)$ induce
operators
$\partial_i : C \otimes H^{\otimes n} \to C \otimes H^{\otimes n-1}$ 
and 
$\sigma_i : C \otimes H^{\otimes n} \to C \otimes H^{\otimes n+1}$,
which are given for $x^0Ê\in C$ and $x^1, \ldots, x^n \in H$ by
\begin{equation*}
\partial_i(x^0 \otimes x^1 \otimes \cdots \otimes x^n) = 
\begin{cases}
x^0 x^1 \otimes x^2 \otimes \cdots \otimes x^n 
& \text{for $i=0$} \, ,\\
x^0 \otimes x^1 \otimes \cdots \otimes x^i x^{i+1} \otimes \cdots \otimes x^n & 
\text{for $0 < i < n$} \, ,\\
x^0 \otimes x^1 \otimes \cdots \otimes x^{n-1} \, \varepsilon(x^n) & 
\text{for $i=n$} \, ,
\end{cases}
\end{equation*}
and 
\begin{equation*}
\sigma_i(x^0 \otimes x^1 \otimes \cdots \otimes x^n) =
x^0 \otimes  \cdots \otimes x^i \otimes 1 \otimes \cdots \otimes x^n 
\quad \text{for $i=0,1, \ldots , n$} \, .
\end{equation*}
It is easy to check that these operators are coalgebra morphisms.
Therefore, $\Gamma_*(H,C)$ is a simplicial coalgebra.
As in Section~\ref{Sw-homology}, we can apply the
functor $F$ to it. The resulting object 
$$F\bigl( \Gamma_*(H,C) \bigr)$$
is a simplicial commutative Hopf algebra.

If in addition $H$ and $C$ are cocommutative, then $\Gamma_*(H,C)$
is a simplicial cocommutative coalgebra and therefore
$F(\Gamma_*(H,C))$ is a simplicial object in the abelian category
of commutative cocommutative Hopf algebras. 
We define the \emph{Sweedler homology $\H^{\SW}_*(H,C)$ of $H$ 
with coefficients in the $H$-module coalgebra~$C$} 
as the homology of the chain complex associated 
to~$F(\Gamma_*(H,C))$:
\begin{equation}
\H^{\SW}_*(H,C) = H_* \bigl( F\bigl( \Gamma_*(H,C) \bigr)\bigr) \, .
\end{equation}
The resulting homology groups are 
\emph{commutative cocommutative Hopf algebras} over~$k$.

When $C=k$ is the trivial $H$-module coalgebra, then
we recover the homology groups $\H^{\SW}_*(H)$ 
of Definition~\ref{HS-hom}:
$$\H^{\SW}_*(H,k) = \H^{\SW}_*(H)\, .$$

\section{Predualization of Lazy Cocycles}\label{predual}

In this section we perform some preliminary basic 
constructions necessary for defining
the lazy homology of a Hopf algebra. The idea is to ``predualize"
the laziness conditions of Section~\ref{lazycohom}.
 
\subsection{The lazy quotient~$C^{[1]}$}\label{lazy-quot1}
Given a coalgebra~$C$, we denote by~$C^{[1]}$  
the quotient of~$C$ by the subspace spanned by the elements
\begin{equation*}
\varphi(x_1)\, x_2-\varphi(x_2)Ê\, x_1 \, ,
\end{equation*} 
for all $x \in C$ and $\varphi \in C^* = \Hom(C,k)$.
The class of $x \in C$ in $C^{[1]}$ will be denoted by~$\underline{x}$.

\begin{proposition}\label{C1coalgebra}
There is a unique coalgebra structure on $C^{[1]}$ such that the projection $C \to C^{[1]}$
is a coalgebra morphism.
Furthermore, we have the 
following ``strong-cocommutativity'' identities:
\begin{equation}\label{strongcomm}
\underline{x_1} \otimes x_2 = \underline{x_2} \otimes x_1 \in C^{[1]} \otimes C
\end{equation}
for all $x\in C$.
\end{proposition}

By definition,
the coproduct $\Delta$ and the counit $\varepsilon$ of~$C^{[1]}$
are given for all $x\in C$ by
\begin{equation*}
\Delta(\underline{x}) = \underline{x_1} \otimes \underline{x_2}  \quad\text{and}\quad
\varepsilon(\underline{x}) = \varepsilon(x)
\end{equation*}
for all $x\in C$.

\begin{proof}
 The coproduct of $C$ induces a linear map
$\Delta' : C \to C^{[1]} \otimes C^{[1]}$ defined by
$\Delta'(x) =  \underline{x_1} \otimes \underline{x_2}$ ($x\in C$).
For $x \in C$ and $\varphi \in C^*$, 
\begin{eqnarray*}
\Delta'(\varphi(x_1) \, x_2) & = & \varphi(x_1) \, \underline{x_2} \otimes \underline{x_3}
=  \varphi(x_2) \, \underline{x_1} \otimes \underline{x_3} \\
& = &  \underline{x_1} \otimes  \varphi(x_2) \, \underline{x_3}
=  \underline{x_1} \otimes \underline{x_2} \, \varphi(x_3) \\
& = & \Delta'(\varphi(x_2) \, x_1) \, .
\end{eqnarray*}
Thus $\Delta'$ induces a linear map $\Delta : C^{[1]} \to C^{[1]} \otimes C^{[1]}$
such that $\Delta(\underline{x}) = \underline{x_1} \otimes \underline{x_2} $ 
for all $x\in C$.
The map~$\Delta$ is clearly coassociative. 
The counit of~$C^{[1]}$ is defined similarly and 
we obtain the desired coalgebra structure on~$C^{[1]}$. 
For $\varphi \in C^*$ and $x \in C$, we have
$$(\id \otimes \, \varphi) (\underline{x_1} \otimes x_2) = 
\varphi(x_2) \, \underline{x_1}= \varphi(x_1) \, \underline{x_2}
= (\id \otimes \, \varphi) (\underline{x_2} \otimes x_1) \, .$$
This yields~\eqref{strongcomm}.
\end{proof}

\begin{corollary}\label{C1cocomm}
The coalgebra $C^{[1]}$ is cocommutative.
\end{corollary}

\begin{proof}
This is a consequence of the second assertion in Proposition~\ref{C1coalgebra}.
\end{proof}

\begin{definition}
The coalgebra~$C^{[1]}$ is called the lazy quotient of~$C$.
\end{definition}

Observe that the coalgebra $C^{[1]} = C$ if $C$ is cocommutative.
The definition of the  lazy quotient
is motivated by the following result. 

\begin{proposition}\label{C1motivation}
Let $C$ be a coalgebra and $R$ a commutative algebra.
For any~$\varphi \in \Reg_{\ell}(C,R)$
there is a unique element $\underline{\varphi} \in \Reg(C^{[1]},R)$
such that $\underline{\varphi}(\underline{x}) = \varphi(x)$ for all $x \in C$. 
There are group isomorphisms
\begin{equation*}
\Reg_{\ell}(C,R) \cong \Reg(C^{[1]},R) \cong \Alg(F(C^{[1]}),R) \, .
\end{equation*}
\end{proposition}

\begin{proof}
If $\psi \in C^*$ and $x \in H$, then by~\eqref{lazy1},
\begin{align*}
\varphi\bigl( \psi(x_1) \, x_2-\psi(x_2) \, x_1 \bigr) 
& = \varphi(x_2) \, \psi(x_1) - \varphi(x_1) \, \psi(x_2) \\
& = \varphi(x_2) \, \psi(x_1) - \varphi(x_2)  \, \psi(x_1) \\
& = 0 \, .
\end{align*}
Hence $\varphi$ induces the announced linear map $\underline{\varphi}$. 
If $\varphi, \psi \in \Reg_{\ell}(C,R)$,
then 
$$\underline{\varphi * \psi} = \underline{\varphi}*\underline{\psi}\, . $$
It follows that $\underline{\varphi}$ is convolution invertible, its
inverse being equal to~$\underline{\varphi^{-1}}$, 
where $\varphi^{-1}$ is the convolution inverse of~$\varphi$.
The map 
$$\Reg_{\ell}(C,R) \to \Reg(C^{[1]},R)$$
defined by
$\varphi \mapsto \underline{\varphi}$ is a homomorphism; 
it is clearly injective.
 
For $\varphi \in \Reg(C^{[1]},R)$, define $\varphi_0 \in \Reg(C,R)$
by $\varphi_0(x) = \varphi(\underline{x})$ for all $x\in C$. 
Using~\eqref{strongcomm}, we obtain
\begin{align*}
\psi\bigl(\varphi_0(x_1) \, x_2-\varphi_0(x_2) \, x_1 \bigr) 
& = \psi \bigl(\varphi(\underline{x_1}) \, x_2-\varphi(\underline{x_2}) \, x_1 \bigr) \\
& = \psi \bigl(\varphi(\underline{x_2}) \, x_1-\varphi(\underline{x_2}) \, x_1 \bigr) = 0
\end{align*}
for all $\psi \in C^*$ and $x \in C$.
It follows that $\varphi_0 \in \Reg_{\ell}(C,R)$
and, since $\underline{\varphi_0}= \varphi$, we obtain the first isomorphism.
The second isomorphism 
follows from Proposition~\ref{Alg-Reg}.
\end{proof}

\subsection{The lazy quotient~$H^{[2]}$}\label{lazy-quot2}

Given a Hopf algebra~$H$, we denote by~$H^{[2]}$ 
the quotient of $H\otimes H$ by the subspace spanned by the elements
\begin{equation}\label{H2-def}
\varphi(x_2y_2) \, x_1\otimes y_1 - \varphi(x_1y_1) \, x_2\otimes y_2
\end{equation} 
for all $x,y \in H$ and $\varphi \in H^* = \Hom(H,k)$.
The class of $x \otimes y \in H\otimes H$ in~$H^{[2]}$
will be denoted by~$\widetilde{x \otimes y}$.

\begin{proposition}\label{H2coalgebra}
There is a unique coalgebra structure on $H^{[2]}$ such that the projection $H \otimes H \to H^{[2]}$
is a coalgebra morphism.
Furthermore, we have the ``lazy-cocommutativity" identities
\begin{equation}\label{lazycomm}
\widetilde{x_1 \otimes y_1} \otimes x_2y_2 = 
 \widetilde{x_2 \otimes y_2} \otimes x_1y_1 \in H^{[2]} \otimes H
\end{equation}
for all $x,y \in H$.
\end{proposition}

By definition,
the coproduct $\Delta$ and the counit $\varepsilon$ of~$H^{[2]}$
are given for all $x,y\in H$ by
\begin{equation*}
\Delta(\widetilde{x \otimes y}) 
= \widetilde{x_1 \otimes y_1} \otimes \widetilde{x_2 \otimes y_2}
\quad\text{and}\quad
\varepsilon(\widetilde{x \otimes y})= \varepsilon(x) \, \varepsilon(y) \, .
\end{equation*}

\begin{proof}
The coproduct of $H \otimes H$ induces a linear map
$$\Delta' : H \otimes H  \to H^{[2]} \otimes H^{[2]}$$
defined by 
$\Delta'(x \otimes y) 
=  \widetilde{x_1 \otimes y_1} \otimes \widetilde{x_2 \otimes y_2}$
($x,y \in H$).
For $x,y \in H$ and $\varphi \in H^*$, 
\begin{align*}
\Delta' \bigl(\varphi(x_1y_1) \, x_2 \otimes y_2 \bigr) 
& = \varphi(x_1y_1) \, \widetilde{x_2 \otimes y_2} \otimes  \widetilde{x_3 \otimes y_3} \\
& =  \varphi(x_2y_2) \, \widetilde{x_1\otimes y_1} \otimes \widetilde{x_3 \otimes y_3} \\
& =  \widetilde{x_1\otimes y_1} \otimes \varphi(x_2y_2) \, \widetilde{x_3 \otimes y_3} \\
& =  \widetilde{x_1\otimes y_1} \otimes \varphi(x_3y_3) \, \widetilde{x_2 \otimes y_2} \\
& =  \varphi(x_3y_3) \, \widetilde{x_1\otimes y_1} \otimes \widetilde{x_2 \otimes y_2}\\
& = \Delta'(\varphi(x_2y_2) \, x_1 \otimes y_1) \, .
\end{align*}
Thus $\Delta'$ induces the announced  linear map $\Delta : H^{[2]} \rightarrow H^{[2]} \otimes H^{[2]}$, 
which is clearly coassociative. 
The counit is defined similarly and we obtain the desired coalgebra structure.
For $\varphi \in H^*$ and $x,y \in H$, 
\begin{align*}
(\id \otimes \, \varphi) (\widetilde{x_1 \otimes y_1} \otimes x_2y_2) 
& =  \varphi(x_2y_2) \, \widetilde{x_1 \otimes y_1} \\
& = \varphi(x_1y_1) \, \widetilde{x_2 \otimes y_2} \\
&  = (\id \otimes \, \varphi) (\widetilde{x_2 \otimes y_2} \otimes x_1y_1) \, .
\end{align*}
This yields~\eqref{lazycomm}.
\end{proof}

\begin{definition}
The coalgebra~$H^{[2]}$
is called the $2$-lazy quotient associated to the Hopf algebra~$H$.
\end{definition}

Observe that $H^{[2]} = H$ if $H$ is cocommutative.
Proposition~\ref{C1motivation} has the following counterpart.

\begin{proposition}\label{H2motivation}
Let $H$ be a Hopf algebra and $R$ a commutative algebra.
For any $\sigma \in  \reg_{\ell}^2(H,R)$
there is a unique element $\underline{\sigma} \in \Reg(H^{[2]},R)$
such that $\underline{\sigma}(\widetilde{x \otimes y})= \sigma(x \otimes y)$
for all $x,y \in H$.
There are group isomorphisms
\begin{equation*}
\reg_{\ell}^2(H,R) \cong \Reg(H^{[2]},R) \cong \Alg(F(H^{[2]}),R) \, .
\end{equation*} 
\end{proposition}

\begin{proof}
The proof follows the same lines as that of Proposition~\ref{C1motivation}. 
Details are left to the reader. 
The algebra morphism $\tilde{\underline{\sigma}}: F(H^{[2]}) \to R$
associated to~$\sigma$ is given by
$$\tilde{\underline{\sigma}} \bigl( t(\widetilde{x\otimes y}) \bigr) = \sigma(x \otimes y)$$
for all $x,y \in H$.
\end{proof}

\section{The First Lazy Homology Hopf Algebra}\label{lazyhom1}

In this section we define the first lazy homology of a Hopf algebra;
it will turn out to be a commutative cocommutative Hopf algebra.

\subsection{The Hopf algebra $\H^{\ell}_1(H)$}\label{HL1a}

To a Hopf algebra $H$ we associate the coalgebra~$H^{[1]}$,
as defined in Section~\ref{lazy-quot1}, 
and the commutative Hopf algebra~$F(H^{[1]})$,
where $F$ is the functor introduced in Section~\ref{Takeuchi}.
We denote by~$\H^{\ell}_1(H)$ the quotient of~$F(H^{[1]})$ 
by the ideal generated by all elements of the form
$$t(\underline{xy}) - t(\underline{x}) \, t(\underline{y}) \, $$
where $x, y \in H$.
The class of an element $a \in F(H^{[1]})$ in $\H^{\ell}_1(H)$
is denoted by~$a^\dagger$.
By definition, 
\begin{equation}\label{t-mult}
t(\underline{xy})^\dagger = t(\underline{x})^\dagger \, t(\underline{y})^\dagger
\end{equation}
for all $x,y \in H$.

\begin{lemma}\label{tbar-mult}
For all $x,y \in H$,
\begin{equation*}
{t}^{-1}(\underline{xy})^\dagger
= {t}^{-1}(\underline{x})^{\dagger} \, {t}^{-1}(\underline{y})^\dagger \, .
\end{equation*}
\end{lemma}

\begin{proof}
Using \eqref{ttbar}, \eqref{t-mult}, and the relations satisfied by counits, we obtain
\begin{align*}
{t}^{-1}(\underline{xy})^\dagger
& = {t}^{-1}(\underline{x_1y_1})^\dagger \, \varepsilon(x_2y_2)
= {t}^{-1}(\underline{x_1y_1})^\dagger \, \varepsilon(x_2) \, \varepsilon(y_2) \\
& = {t}^{-1}(\underline{x_1y_1})^\dagger \,
t(\underline{x_2})^{\dagger} \, t(\underline{y_2})^\dagger \, 
{t}^{-1}(\underline{x_3})^\dagger \, 
{t}^{-1}(\underline{y_3})^\dagger \\
& =   {t}^{-1}(\underline{x_1y_1})^\dagger \, 
t(\underline{x_2y_2})^\dagger \, 
{t}^{-1}(\underline{x_3})^\dagger \, {t}^{-1}(\underline{y_3})^\dagger \\
& = \varepsilon(x_1y_1) \,
{t}^{-1}(\underline{x_2})^\dagger \, 
{t}^{-1}(\underline{y_2})^\dagger
= \varepsilon(x_1) \, \varepsilon(y_1) \,
{t}^{-1}(\underline{x_2})^\dagger \, 
{t}^{-1}(\underline{y_2})^\dagger \\
& = {t}^{-1}(\underline{x})^\dagger \, {t}^{-1}(\underline{y})^\dagger \, .
\end{align*}
\end{proof}

\begin{proposition}\label{HL1-Hopf}
There is a unique Hopf algebra structure on~$\H^{\ell}_1(H)$
such that the natural projection $F(H^{[1]}) \to \H^{\ell}_1(H)$
is a Hopf algebra morphism.
The Hopf algebra $\H^{\ell}_1(H)$ is commutative and cocommutative. 
\end{proposition}

It follows from~\eqref{Hopf-on-Takeuchi} that
the coproduct~$\Delta$, the counit~$\varepsilon$, 
and the antipode~$S$ of~$\H^{\ell}_1(H)$ are given for all $x\in H$ by
\begin{align*}
\Delta \bigl( t(\underline{x})^\dagger \bigr) 
& =  t(\underline{x_1})^\dagger \otimes t(\underline{x_2})^\dagger \, , \quad
\Delta \bigl( {t}^{-1}(\underline{x})^\dagger \bigr) 
=  {t}^{-1}(\underline{x_2})^\dagger \otimes {t}^{-1}(\underline{x_1})^\dagger \, , \\ 
\varepsilon \bigl( t(\underline{x})^\dagger \bigr)  
& = \varepsilon \bigl( {t}^{-1}(\underline{x})^\dagger \bigr) = \varepsilon(x)\, , \\
S \bigl( t(\underline{x})^\dagger \bigr) & = {t}^{-1}(\underline{x})^\dagger \, , \quad
S \bigl( {t}^{-1}(\underline{x})^\dagger \bigr) = t(\underline{x})^\dagger \, .
\end{align*}  

\begin{proof}
The coproduct of $F(H^{[1]})$ induces an algebra morphism
\begin{equation*}
\Delta' : F(H^{[1]})  \to \H^{\ell}_1(H) \otimes \H^{\ell}_1(H)
\end{equation*}
such that
$\Delta'(t(\underline{x}))
=  t(\underline{x_1})^\dagger \otimes t(\underline{x_2})^\dagger$
and
$\Delta'({t}^{-1}(\underline{x})) =
{t}^{-1}(\underline{x_2})^\dagger \otimes {t}^{-1}(\underline{x_1})^\dagger$
for all $x\in H$.
Using~\eqref{t-mult}, we have
\begin{eqnarray*}
\Delta' \bigl( t(\underline{xy}) \bigr) 
& = & t(\underline{x_1y_1})^\dagger \otimes t(\underline{x_2y_2})^\dagger \\
& = & t(\underline{x_1})^\dagger t(\underline{y_1})^\dagger \otimes t(\underline{x_2})^\dagger t(\underline{y_2})^\dagger \\
& = & \Delta' \bigl( t(\underline{x}) \bigr) \, \Delta' \bigl( t(\underline{y}) \bigr) \\
& = & \Delta' \bigl( t(\underline{x})t(\underline{y}) \bigr)
\end{eqnarray*}
for all $x,y \in H$.
Thus $\Delta'$ induces the desired coproduct. 
The construction of the counit and the antipode are similar and left to the reader 
(use Lemma~\ref{tbar-mult}). 
The Hopf algebra $\H^{\ell}_1(H)$ inherits the commutativity of~$F(H^{[1]})$
and the cocommutativity of~$H^{[1]}$. 
\end{proof}

\begin{definition}
The first lazy homology Hopf algebra of~$H$
is the commutative cocommutative Hopf algebra~$\H^{\ell}_1(H)$. 
\end{definition}

The following result relates~$\H^{\ell}_1(H)$ to the
lazy cohomology group $\H_{\ell}^1(H,R)$ of Definition~\ref{lazycoh1}.

\begin{theorem}\label{UCT1}
For any Hopf algebra~$H$ and any commutative algebra~$R$
there is a group isomorphism
$$\H_{\ell}^1(H,R)  \cong \Alg(\H^{\ell}_1(H),R) \, .$$
\end{theorem}

\begin{proof}
Let $\mu \in \H_{\ell}^1(H,R)$. Since $\mu \in \Reg_{\ell}(H,R)$, 
we know from Propositions~\ref{Alg-Reg} and~\ref{C1motivation} 
that $\mu$ induces an algebra morphism $\mu^* : F(H^{[1]}) \to R$.
The morphism~$\mu^*$ vanishes on the defining ideal of~$\H^{\ell}_1(H)$, 
since $\mu$ is an algebra morphism. 
We thus obtain a map
\begin{equation*}
\H_{\ell}^1(H,R)  \to \Alg(\H^{\ell}_1(H),R) \, ,  \; \mu \mapsto \mu^* \, ,
\end{equation*}
where $\mu^*(t(\underline{x})^\dagger) = \mu(x)$
and $\mu^*({t}^{-1}(\underline{x})^\dagger) = \mu(S(x))$
for all $x\in H$. The map $ \mu \mapsto \mu^*$
is obviously injective and is a homomorphism.

Any map $\chi \in \Alg(\H^{\ell}_1(H),R)$ induces a linear
map $\chi_0 : H \to R$ defined by 
$\chi_0(x) = \chi(t(\underline{x})^\dagger)$ for all $x\in H$. 
By Proposition~\ref{C1motivation} and its proof,
$\chi_0$~belongs to~$\Reg_{\ell}(H,R)$.
For $x,y \in H$, 
$$\chi_0(xy) = \chi(t(\underline{xy})^\dagger) 
= \chi \bigl( t(\underline{x})^\dagger \bigr) \, \chi \bigl( t(\underline{y})^\dagger \bigr)
= \chi_0(x) \, \chi_0(y) \, ;$$
hence, $\chi_0 \in \H_{\ell}^1(H,R)$. 
In view of~$\chi_0^*=\chi$, this concludes the proof.
\end{proof}

\subsection{An alternative description of~$\H^{\ell}_1(H)$}\label{HL1b}

We now give another, more direct, description of~$\H^{\ell}_1(H)$. 
Let $H^{\langle 1\rangle }$ be the quotient of~$H$ 
by the two-sided ideal~$I$ generated by the elements
\begin{equation}\label{idealH1}
\varphi(x_1) \, x_2 - \varphi(x_2) \, x_1\, ,
\end{equation}
where $x \in H$ and $\varphi \in H^* = \Hom(H,k)$. 
The canonical projection $H \to H^{\langle 1\rangle }$ obviously
factors through the quotient $H^{[1]}$ defined in Section~\ref{lazy-quot1}.

The proof of the following proposition is similar to 
that of Proposition~\ref{C1coalgebra} and is left to the reader.

\begin{proposition}
There is a unique Hopf algebra structure on~$H^{\langle 1\rangle}$
such that the natural projection $H \to H^{\langle 1\rangle}$
is a Hopf algebra morphism.
The Hopf algebra~$H^{\langle 1 \rangle}$ is cocommutative.
\end{proposition}

The Hopf algebra $H^{\langle 1 \rangle}$ has the following property.

\begin{proposition}
For any commutative algebra~$R$,
the group $\Alg(H^{\langle 1 \rangle},R)$ consists of 
the elements of $\Alg(H,R)$ commuting in~$\Hom(H,R)$
with all elements of $\Hom(H,k)\, 1_R$.
\end{proposition}

\pf
An element $\psi \in \Alg(H,R)$ factors through~$H^{\langle 1 \rangle}$ if and only if
it vanishes on the elements~\eqref{idealH1}, i.e., if and only if
$\psi * \varphi = \varphi * \psi$ for all $\varphi \in H^*$.
\epf

Consider the case where $R$ is the ground field~$k$. It is well known that
if $k$ is algebraically closed of characteristic zero and $H$ is commutative,
then $\Alg(H,k)$ separates the points of~$H$. 
We can then sharpen the previous proposition and state that under the previous conditions,
$\Alg(H^{\langle 1 \rangle},k)$ coincides with the center of the group~$\Alg(H,k)$.

For any Hopf algebra $H$, we denote
by $H_\ab$ the quotient of $H$ by the two-sided ideal generated 
by all commutators $[x,y] = xy- yx$ of elements $x,y$ of~$H$.
Since
$$\Delta([x,y]) = [x_1,y_1] \otimes x_2y_2 + y_1 x_1 \otimes [x_2,y_2] \, ,$$
$\varepsilon([x,y]) = 0$, and $S([x,y]) = - [S(x), S(y)]$,
the Hopf structure on~$H$ induces a commutative
Hopf algebra structure on~$H_\ab$.
We call~$H_\ab$ the \emph{abelianization} of~$H$. 
We are now ready to give the promised alternative description of~$\H^{\ell}_1(H)$.

\begin{proposition}\label{HL1-prop}
For any Hopf algebra~$H$ there is 
a Hopf algebra isomorphism
\begin{equation}\label{alternateHL1}
\H^{\ell}_1(H) \cong (H^{\langle 1 \rangle})_\ab \, .
\end{equation}
\end{proposition}

\begin{proof}
 The composition of the canonical projections
 $H^{[1]} \to H^{\langle 1 \rangle} \to (H^{\langle 1 \rangle})_\ab$
 induces a surjective coalgebra morphism
 $H^{[1]} \to (H^{\langle 1 \rangle})_\ab$, 
 which by Proposition~\ref{F-univprop} induces
 a Hopf algebra morphism $p : F(H^{[1]}) \to (H^{\langle 1 \rangle})_\ab$.
 If we denote the class of an element $x \in H$ 
 in~$(H^{\langle 1 \rangle})_\ab$ by~$x^\circ$, 
 then $p(t(\underline{x})) = x^\circ$ and
 $$p\bigl(t(\underline{xy})\bigr) 
 = (xy)^\circ = x^\circ \, y^\circ 
 = p\bigl(t(\underline{x})\bigr) \, p\bigl(t(\underline{y})\bigr) 
 = p\bigl(t(\underline{x})\, t(\underline{y})\bigr)$$
 for all $x,y \in H$. Hence,
 $p$ induces a (surjective) Hopf algebra morphism 
 $\H^{\ell}_1(H) \to (H^{\langle 1 \rangle})_\ab$. 
 
An inverse isomorphism is constructed as follows.
By~\eqref{strongcomm}, 
the Hopf algebra morphism 
$$H \to \H^{\ell}_1(H)\, , \; x \mapsto t(\underline{x})^\dagger$$
induces a Hopf algebra morphism
$H^{\langle 1 \rangle} \to \H^{\ell}_1(H)$.
Since by Proposition~\ref{HL1-Hopf}
the Hopf algebra~$\H^{\ell}_1(H)$ is commutative, 
the latter morphism induces a Hopf algebra morphism 
$$(H^{\langle 1\rangle})_\ab \to \H^{\ell}_1(H) \, , \; 
x^\circ \mapsto t(\underline{x})^\dagger \, ,$$
which is the required inverse isomorphism.
\end{proof}

\begin{corollary}
If $H$ is a finite-dimensional Hopf algebra, then so is~$\H^{\ell}_1(H)$.
\end{corollary}

Proposition~\ref{HL1-prop} also allows us to compute the first lazy homology
of the Hopf algebra $k^G$ of functions on a finite group~$G$.

\begin{proposition}\label{HL1k^G}
For any finite group~$G$, we have
\begin{equation*}
\H^{\ell}_1(k^G) \cong k^{Z(G)} \, ,
\end{equation*}
where $k^{Z(G)}$ is the Hopf algebra of functions on the center of~$G$.
\end{proposition}

\pf
For each $g\in G$, let $e_g$ be the function on~$G$ that is zero everywhere, except at the point~$g$,
where it takes the value~$1$. The elements $(e_g)_{g\in G}$ form a basis of~$H = k^G$.
Let $(\psi_g)_{g\in G}$ be the dual basis. Then $H^{\langle 1 \rangle}$ is the quotient of~$H$ by
the ideal generated by
$$\sum_{a\in G}\, \psi_{h^{-1}}(e_a) \, e_{a^{-1}g} - \sum_{a\in G}\, e_{ga^{-1}} \, \psi_{h^{-1}}(e_a)$$
for all $g,h \in G$.
The first sum reduces to~$e_{hg}$, while the second sum reduces to~$e_{gh}$.
Therefore, $H^{\langle 1 \rangle}$ is the quotient of~$H$ by
the ideal generated by $w_{g,h} = e_{hg} - e_{gh}$ for all $g,h \in G$.
Now, $w_{g,h}= 0$ if $g$ belongs to the center of~$G$. If $g$ is not in the center, then
there is $h$ such that $hgh^{-1} \neq g$; for such an element~$h$, we have
$$e_g = e_g^2 = e_g^2 - e_g e_{hgh^{-1}} = e_g w_{g,hgh^{-1}}\, ,$$
which shows that the image of~$e_g$ in~$H^{\langle 1 \rangle}$ is zero.
From this one easily deduces that~$H^{\langle 1 \rangle} \cong k^{Z(G)}$.
The conclusion follows from Proposition~\ref{HL1-prop} and the commutativity of~$H$,
hence of~$H^{\langle 1 \rangle}$.
\epf

\subsection{An interpretation of~$\H^{\ell}_1(H)$ as a homology group}\label{HL1-hom}

We could have defined~$\H^{\ell}_1(H)$ as the right-hand side of~\eqref{alternateHL1}.
Nevertheless, the definition we gave in Section~\ref{HL1a}
will allow us to present~$\H^{\ell}_1(H)$ as a kind of homology group
(see Proposition~\ref{HS1=HL1} below).
Such a presentation will be used in the next section to define
a lazy homology group~$\H^{\ell}_2(H)$.

We first construct a Hopf algebra morphism
generalizing the differential~$d^{\SW}_2$ of Lemma~\ref{d2-cocomm}.

\begin{lemma}\label{d2}
For any Hopf algebra~$H$ there is a Hopf algebra morphism
\begin{equation*}
d_2 : F(H^{[2]}) \to F(H^{[1]}) 
\end{equation*}
such that for all $x,y \in H$,
\begin{equation}\label{d2c}
d_2 \bigl(t(\widetilde{x \otimes y})\bigr)
= t(\underline{x_1}) \, t(\underline{y_1}) \, t^{-1}(\underline{x_2y_2})  \, .
\end{equation}
Moreover,
\begin{equation}\label{d2normal}
d_2(a_1) \otimes a_2 = d_2(a_2) \otimes a_1
\end{equation} 
for all $a \in F(H^{[2]})$.
\end{lemma}

\begin{proof}
Consider the map $d'_2 : H\otimes H \to F(H^{[1]})$ given by
$$d'_2(x \otimes y)
= t(\underline{x_1}) \, t(\underline{y_1}) \, t^{-1}(\underline{x_2y_2})$$
for all $x,y \in H$.
Let us show that $d'_2$ is a coalgebra morphism. 
On one hand, by~\eqref{Hopf-on-Takeuchi}, we have
\begin{align*}
\Delta\bigl(d'_2(x \otimes y) \bigr)
& = \Delta\bigl(t(\underline{x_1}) \,   t(\underline{y_1}) \, 
t^{-1}(\underline{x_2y_2}) \bigr) \\
& = t(\underline{x_1}) \, t(\underline{y_1}) \, t^{-1}(\underline{x_4y_4}) 
\otimes t(\underline{x_2}) \, t(\underline{y_2}) \, t^{-1}(\underline{x_3y_3}) \, .
\end{align*}
On the other hand,
\begin{align*}
(d'_2 \otimes d'_2) \bigl( \Delta (x\otimes y)\bigr)
& = (d'_2 \otimes d'_2) ( x_1 \otimes y_1 \otimes x_2 \otimes y_2) \\
& = t(\underline{x_1}) \, t(\underline{y_1}) \, t^{-1}(\underline{x_2y_2}) 
\otimes t(\underline{x_3}) \, t(\underline{y_3}) \, t^{-1}(\underline{x_4y_4}) \, .
\end{align*}
These two expressions are equal in view of the 
strong-cocommutativity identities~\eqref{strongcomm}
of~$H^{[1]}$.
It is easy to check that
$\varepsilon \bigl(d'_2(x \otimes y) \bigr) = \varepsilon(x)\, \varepsilon(y)$
for all $x,y \in H$. We have thus shown that $d'_2$ is a coalgebra morphism. 

We next claim that $d'_2$ factors through~$H^{[2]}$.
In view of~\eqref{H2-def} we have to check that
$$d'_2 \bigl( \varphi(x_2y_2) \, x_1\otimes y_1 
- \varphi(x_1y_1) \, x_2\otimes y_2 \bigr) = 0$$
for all $\varphi \in H^*$ and $x,y \in H$.
Now the left-hand side is equal to
$$\varphi(x_3y_3) \, t(\underline{x_1}) \, 
t(\underline{y_1}) \, t^{-1}(\underline{x_2y_2})
- \varphi(x_1y_1) \, t(\underline{x_2}) \, 
t(\underline{y_2}) \, t^{-1}(\underline{x_3y_3})  \, ,$$
which vanishes in view of~\eqref{strongcomm}.
By Proposition~\ref{F-univprop},
the induced coalgebra morphism $H^{[2]} \to F(H^{[1]})$
induces the desired Hopf algebra morphism 
$d_2 : F(H^{[2]}) \to F(H^{[1]})$.

Let us now prove~\eqref{d2normal}.
Since $d_2$ is a Hopf algebra morphism, it is enough to prove the required identity 
for~$a = t(\widetilde{x \otimes y})$, where $x,y \in H$.
We have
\begin{align*}
d_2 \bigl(t(\widetilde{x \otimes y})_1 \bigr) \otimes t(\widetilde{x \otimes y})_2 
& = d_2\bigl(t(\widetilde{x_1 \otimes y_1}) \bigr)
\otimes t(\widetilde{x_2 \otimes y_2})\\
& = t(\underline{x_1}) \, t(\underline{y_1}) \, {t}^{-1}(\underline{x_2y_2}) 
\otimes t(\widetilde{x_3 \otimes y_3})\\
& = t(\underline{x_1}) \, t(\underline{y_1}) \, {t}^{-1}(\underline{x_3y_3}) \otimes
t(\widetilde{x_2 \otimes y_2}) \\
& = t(\underline{x_2}) \, t(\underline{y_2}) \, {t}^{-1}(\underline{x_3y_3}) \otimes
t(\widetilde{x_1 \otimes y_1}) \\
& = d_2 \bigl(t(\widetilde{x_2 \otimes y_2}) \bigr) \otimes t(\widetilde{x_1 \otimes y_1})\\
& = d_2 \bigl(t(\widetilde{x \otimes y})_2 \bigr) \otimes t(\widetilde{x \otimes y})_1 \, .
\end{align*}
We have used the lazy-cocommutativity identities~\eqref{lazycomm} in~$H^{[2]}$ 
for the third equality above and the strong-cocommutativity identities~\eqref{strongcomm}
in~$H^{[1]}$ for the fourth equality.
\end{proof}

As a consequence of Lemma~\ref{d2} and Section~\ref{HKer},
we obtain the following.

\begin{corollary}\label{HKerd2}
The Hopf algebra morphism $d_2$ is normal and 
\begin{equation*}
\HKer(d_2) = \{ a \in F(H^{[2]}) \ | \ d_2(a_1) \otimes a_2 = 1 \otimes a\}
= \{ a \in F(H^{[2]}) \ | \ a_1 \otimes d(a_2) = a \otimes 1\} \, .
\end{equation*}
\end{corollary}

By Lemma~\ref{d2}, the image $\Im(d_2)$ of~$d_2$
is a Hopf subalgebra of the commutative 
cocommutative Hopf algebra~$F(H^{[1]})$,
and it is cocommutative.
Using the convention of Section~\ref{HKer} (see~\eqref{quotient}),
we can define the quotient $F(H^{[1]}) \quot \Im(d_2)$
as the quotient of~$F(H^{[1]})$ by the ideal generated
by the augmentation ideal~$\Im(d_2)^+$ of~$\Im(d_2)$.

We now express $\H^{\ell}_1(H)$ in terms of this homology-like quotient.

\begin{proposition}\label{HS1=HL1}
For any Hopf algebra~$H$, there is an isomorphism of Hopf algebras
$$\H^{\ell}_1(H)  \cong F(H^{[1]}) \quot \Im(d_2)  
= {\HKer}(\varepsilon) \quot \Im(d_2)\, .$$
If in addition $H$ is cocommutative, then 
$$\H^{\ell}_1(H) = \H^{\SW}_1(H) \, . $$
\end{proposition}

\begin{proof}
In order to prove the first isomorphism,
we check that the ideal generated by~$\Im(d_2)^+$
coincides with the ideal~$I$ of~$F(H^{[1]})$
generated by
$$t(\underline{xy}) - t(\underline{x}) \, t(\underline{y}) \, ,$$
where $x, y\in H$.
By an easy computation one shows that for $x,y \in H$,
\begin{equation}\label{d2inI}
d_2\bigl(t(\widetilde{x \otimes y})\bigr) 
- \varepsilon\bigl(t(\widetilde{x \otimes y})\bigr) 
= - \bigl( t(\underline{x_1y_1}) - t(\underline{x_1}) \, t(\underline{y_1})\bigr) 
\, {t}^{-1}(\underline{x_2y_2}) \, . 
\end{equation}
It follows that 
$d_2(\omega) - \varepsilon(\omega) \in I$
for all $\omega \in F(H^{[2]})$.
Hence, $\Im(d_2)^+ \subset I$.

By~\eqref{ttbar}, Equation~\eqref{d2inI} implies that
$$\bigl[ d_2\bigl(t(\widetilde{x_1 \otimes y_1})\bigr) 
- \varepsilon\bigl(t(\widetilde{x_1 \otimes y_1})\bigr) \bigr]
\, t(\underline{x_2y_2}) 
= - \bigl( t(\underline{xy}) - t(\underline{x}) \, t(\underline{y})\bigr) \, .
$$
From this we deduce the converse inclusion $I \subset \Im(d_2)^+$.

The equality $F(H^{[1]}) / \Im(d_2)  = {\HKer}(\varepsilon)/ \Im(d_2)$
follows from the equality 
$$F(H^{[1]}) = \HKer(\varepsilon: F(H^{[1]}) \to k)\, ,$$
a general fact observed in Section~\ref{HKer}.

If $H$ is cocommutative, then $F(H^{[1]}) = F(H)$,
$F(H^{[2]}) = F(H\otimes H)$, and 
by~\eqref{d2c} and~\eqref{d2a}, 
the map~$d_2$ coincides with the differential~$d^{\SW}_2$ of Lemma~\ref{d2-cocomm}.
This yields the second part of the statement.
\end{proof}

When $H=k[G]$ is a group algebra, Propositions~\ref{HSkG}
and~\ref{HS1=HL1} imply that
\begin{equation}\label{HS1kG}
\H^{\ell}_1(k[G]) = \H^{\SW}_1(k[G]) \cong k[H_1(G,\ZZ)] = k[G_{\ab}] \, ,
\end{equation}
where $G_{\ab}$ is the largest abelian quotient of~$G$.

\subsection{An $\Ext^1$-group}\label{ext}

By Lemma~\ref{d2}, the image $\Im(d_2)$
of the Hopf algebra morphism
$d_2 : F(H^{[2]}) \to F(H^{[1]})$
is a cocommutative Hopf subalgebra of~$F(H^{[1]})$.
Consider the sequence of Hopf algebra morphisms
$$k \longrightarrow \Im(d_2)
\overset{\iota}{\longrightarrow} F(H^{[1]}) 
\overset{\pi}{\longrightarrow} \H^{\ell}_1(H)
\longrightarrow k\, ,$$
where $\iota$ is the embedding of~$\Im(d_2)$ into~$F(H^{[1]})$
and $\pi$ is the natural projection defining~$\H^{\ell}_1(H)$.
This sequence is exact since $\iota$ is injective, $\pi$
is surjective, and $\Ker (\pi) = \Im(d_2)^+ F(H^{[1]})$ by Proposition~\ref{HS1=HL1}.
It follows from Proposition~\ref{Alg-exact} that
for any commutative algebra~$R$, the sequence of groups 
\begin{equation*}
0 \longrightarrow \Alg(\H^{\ell}_1(H),R) \overset{\pi^*}{\longrightarrow}  
\Alg(F(H^{[1]}),R) \overset{\iota^*}{\longrightarrow}  \Alg(\Im(d_2),R)
\end{equation*}
is exact. 
Observe that the groups in the previous sequence are all abelian
since the Hopf algebras involved are cocommutative.

Mimicking the definition of $\Ext^1$ in standard homological algebra,
we pose the following.

\begin{definition}\label{def-ext}
The group $\Ext^1(H,R)$ is the
cokernel of the map~$\iota^*$:
$$\Ext^1(H,R) = 
\Coker \bigl (\iota^* : \Alg(F(H^{[1]}),R) \longrightarrow \Alg(\Im(d_2),R) \bigr) \, .$$
\end{definition}

The group $\Ext^1(H,R)$ is \emph{abelian};
it is clearly contravariant in~$H$
and covariant in~$R$. 
The following is an immediate consequence
of the definition and of Proposition~\ref{Alg-surj}.

\begin{proposition}\label{Ext-zero}
If $k$ is algebraically closed, then $\Ext^1(H,k) = 0$.
\end{proposition}

\begin{rem}
The corresponding $\Ext^1$-groups in the case of a group~$G$
can be computed from the first homology group~$H_1(G,\ZZ)$, 
hence by~\eqref{HS1kG} from~$\H^{\ell}_1(H)$ for $H = k[G]$.
We may wonder whether for an arbitrary Hopf algebra~$H$ 
the group~$\Ext^1(H,R)$ depends only 
on the lazy cohomology Hopf algebra~$\H^{\ell}_1(H)$.
Observe that if $\H^{\ell}_1(H) \cong k$ is trivial, then
$\iota: \Im(d_2) \to F(H^{[1]})$ is an isomorphism and $\Ext^1(H,R)=0$.
\end{rem}

\section{The Cosemisimple Case}\label{cosemisimple}

In this section, assuming that 
the ground field~$k$ is algebraically closed of characteristic zero,
we compute the first lazy homology of a cosemisimple
Hopf algebra~$H$. 
More precisely, we show that $\H_1^\ell(H)$
is isomorphic to the group algebra of the universal abelian grading group 
of the tensor category of $H$-comodules.

As an application, we recover M\"uger's description of the center
of a compact group~\cite{mu} as the group of unitary characters
of the universal abelian grading group of the representation category.
This shows that one can reconstruct the center of such a group from its representation category
(or better, from its fusion semiring). 

We first fix some notation.
Let $H$ be a cosemisimple Hopf algebra and
$\Irr(H)$ be the set of isomorphism classes of simple (irreducible) $H$-comodules.
The isomorphism class of a simple $H$-comodule $V$ is denoted by~$[V]$.
For each $\lambda \in \Irr(H)$, we fix a simple $H$-comodule $V_\lambda$
such that $[V_\lambda]= \lambda$. 

For $\lambda, \mu, \nu \in \Irr(H)$, we write $\nu \prec \lambda \otimes \mu$ when
$V_\nu$ is isomorphic to a subcomodule of $V_\lambda \otimes V_\mu$.

\begin{definition}
The universal abelian grading group~$\Gamma_H$
of the semisimple tensor category of $H$-comodules
is the quotient of the free (multiplicative) abelian group
generated by $\Irr(H)$ modulo the relations
$$\lambda \mu = \nu \ \text{whenever} \ \nu \prec\lambda \otimes \mu \, .$$
\end{definition}

The class of $\lambda \in \Irr(H)$ in~$\Gamma_H$ is denoted by~$|\lambda|$.
The universal abelian grading group of a semisimple tensor category has appeared
(in various degrees of generality) in several independant papers; 
see \cite{bl, mu, gn, pe}. 

We have the following result.

\begin{theorem}\label{H1cos}
Let $H$ be a cosemisimple Hopf algebra over  an algebraically closed field
of characteristic zero. Then there is a Hopf algebra isomorphism
$$\H_1^{\ell}(H) \cong k[\Gamma_H] \, .$$
\end{theorem}

The rest of the subsection is essentially devoted to the proof of Theorem
\ref{H1cos}, with an application to centers of compact groups at the very end.

We need some additional notation.
For each $\lambda \in \Irr(H)$, we fix a basis
$e_1^{\lambda}, \ldots, e_{d_\lambda}^{\lambda}$ of $V_\lambda$ (so that $d_\lambda = \dim(V_\lambda)$)
and elements $x^\lambda_{ij}$ of~$H$ ($1\leq i,j\leq d_\lambda$) such that
$\alpha(e_j^{\lambda})= \sum_{i=1}^{d_\lambda} \, e_i^{\lambda} \otimes x_{ij}^\lambda$, 
where $\alpha$ stands for the $H$-coaction on $V_\lambda$. 
The elements $x_{ij}^\lambda \in H$ ($1 \leq i,j \leq d_\lambda$, $\lambda \in \Irr(H)$) 
are linearly independent and satisfy
\begin{equation}\label{Deltaxij}
\Delta(x_{ij}^\lambda) 
=  \sum_{k=1}^{d_\lambda} \, x_{ik}^\lambda \otimes x_{kj}^\lambda \, , 
\qquad \varepsilon(x_{ij}^\lambda)=\delta_{ij} \, ,
\end{equation}
and we have the coalgebra direct sum (Peter-Weyl decomposition)
$$H = \bigoplus_{\lambda \in \Irr(H)} H_\lambda$$
where $H_\lambda$ is the comatrix subcoalgebra spanned by the 
elements $x_{ij}^\lambda$ ($1 \leq i,j \leq d_\lambda$).  

We first describe the coalgebra $H^{[1]}$.

\begin{lemma}\label{xlambda}
The coalgebra $H^{[1]}$ 
has a basis consisting of the grouplike elements 
$x^\lambda = \underline{x_{11}^\lambda}$ ($\lambda \in \Irr(H)$). 
We have $\underline{x_{ij}^\lambda} = 0$ for all $i \neq j$ and
$\underline{x_{ii}^\lambda} = x^\lambda$ for all $i = 1, \ldots, d_\lambda$.
\end{lemma}

It follows from this lemma and from~\eqref{Deltaxij} that 
each element $x^\lambda \in H^{[1]}$ is grouplike.
Thus, $H^{[1]}$ is a \emph{cosemisimple pointed} coalgebra.

\begin{proof} 
Let $(\psi_{ij}^\lambda)_{ij}$ be the dual basis of the 
basis~$(x_{ij}^\lambda)_{ij}$. 
The following relations hold  in $H^{[1]}$ for any $\varphi \in H^*$:
$$\sum_k \underline{x_{ik}^\lambda} \, \varphi(x_{kj}^\lambda)
= \sum_{k} \underline{x_{kj}^\lambda} \, \varphi(x_{ik}^\lambda)$$ 
For $\varphi=\psi_{ij}^\lambda$, this gives $\underline{x_{ii}^\lambda}=\underline{x_{jj}^\lambda}$.
For $\varphi = \psi_{ii}^\lambda$, this gives $\underline{x_{ij}^\lambda}=0$
if $i \neq j$.
Hence the elements $x^\lambda$, as defined in the statement of the lemma, 
span~$H^{[1]}$. One easily checks that they are linearly independent, using 
the linear forms $\psi^\lambda \in H^*$ defined by
$\psi^\lambda(x_{ij}^\mu) = \delta_{\lambda\mu}\delta_{ij}$.
\end{proof}

It follows that the Hopf algebra $F(H^{[1]})$ is the group algebra of the free abelian 
group generated by~$\Irr(H)$. 

Before giving the proof of Theorem~\ref{H1cos}, we state the following
lemma, whose elementary proof is  left to the reader.

\begin{lemma}\label{elem}
Let $g, h_1, \ldots h_n$ be elements of a group~$G$. Assume
that in the group algebra $k[G]$ we have
$$g= \sum_{i=1}^n \, r_i \, h_i $$
for some positive rational numbers $r_1, \ldots, r_n$.
Then $h_1 = \cdots = h_n=g$. 
\end{lemma}

\begin{proof}[Proof of Theorem~\ref{H1cos}]
The character of an element $\lambda \in \Irr(H)$ is the element
$\chi^\lambda = \sum_{i=1}^{d_\lambda}x_{ii}^\lambda \in H$.
For all $\lambda, \mu \in \Irr(H)$, we have
$$\chi^\lambda \chi^\mu = \sum_{\nu\prec \lambda \otimes \mu}
d_{\lambda\mu}^\nu \, \chi^\nu \, ,$$
where $d_{\lambda\mu}^\nu = \dim \Hom(V_\nu, V_\lambda \otimes V_\mu)$.
By Lemma~\ref{xlambda}, 
$\underline \chi^\lambda = d_\lambda x^\lambda$. This leads
to the following computation in $\H_1^\ell(H)$:
\begin{align*}
t(x^\lambda)^\dagger  \, t(x^\mu)^\dagger 
&= \frac{1}{d_\lambda d_\mu} \, t(\underline{\chi^\lambda})^\dagger \, 
t(\underline{\chi^\mu})^\dagger
=\frac{1}{d_\lambda d_\mu} \, t(\underline{\chi^\lambda\chi^\mu})^\dagger \\
& = \frac{1}{d_\lambda d_\mu} \, \sum_{\nu\prec \lambda \otimes \mu} \, 
d_{\lambda\mu}^\nu \, t(\underline{\chi^\nu})^\dagger 
=  \sum_{\nu\prec \lambda \otimes \mu} \, \frac{d_{\lambda\mu}^\nu}{d_\lambda d_\mu} \, 
t(\underline{\chi^\nu})^\dagger \\
& = \sum_{\nu\prec \lambda \otimes \mu} \, 
\frac{d_{\lambda\mu}^\nu d_\nu}{d_\lambda d_\mu} \, t(x^\nu)^\dagger \, .
\end{align*}
The Hopf algebra $\H_1^\ell(H)$ is a group algebra since it is a quotient of the group
algebra $F(H^{[1]})$ and the elements $t(x^\lambda)^\dagger$ are grouplike elements.
Hence by Lemma~\ref{elem}, if $\nu \prec \lambda \otimes \mu$, then
$$t(x^\lambda)^\dagger \, t(x^\mu)^\dagger = t(x^\nu)^\dagger \, .$$
This shows that there is a Hopf algebra morphism
$k[\Gamma_H] \rightarrow \H_1^\ell(H)$ given by
$$|\lambda| \mapsto t(x^\lambda)^\dagger \, .$$
This Hopf algebra morphism is clearly surjective, and since it is a Hopf algebra morphism
between two group algebras, it is enough to check its injectivity on the grouplike
elements. 
Since the characters of a discrete abelian group separate its points, this is equivalent
to show that the induced injective group homomorphism
\begin{equation}\label{iso-coss}
\H_{\ell}^1(H,k) \cong \Alg(\H^\ell_1(H),k) \longrightarrow  
\widehat{\Gamma_H} = \Hom(\Gamma_H,k^{\times})
\end{equation}
sending $\varphi \in \Alg(\H^\ell_1(H),k)$ to the element
$\varphi' \in \widehat{\Gamma_H}$ defined by
$$\varphi'(|\lambda|) = \varphi(t(x^\lambda)^\dagger)$$ 
($\lambda \in \Irr(H)$), is surjective. 
For $\alpha \in  \widehat{\Gamma_H}$, consider the linear 
form $\alpha_0$ on $H = \bigoplus_{\lambda \in \Irr(H)} \, H_\lambda$ 
defined by  
$$\alpha_0 = \sum_{\lambda \in \Irr(H)} \, \alpha(|\lambda|) \, \varepsilon_{|H_\lambda}\, .$$
It is clear that $\alpha_0$ is an element of ${\Reg}^1_{\ell}(H)$. Let us check that
$\alpha_0$ is an algebra morphism. 
Let $\lambda, \mu \in \Irr(H)$, $1\leq i,j \leq d_\lambda$, $1 \leq k,l \leq d_\mu$. We have 
$x_{ij}^\lambda \, x_{kl}^\mu = \sum_{\nu\prec \lambda \otimes \mu} z_\nu$ for some elements
$z_\nu \in H_\nu$, since by Tannakian reconstruction we have 
$H_\lambda H_\mu \subset \sum_{\nu\prec \lambda \otimes \mu} \, H_\nu$. 
Therefore,
\begin{align*}
\alpha_0(x_{ij}^\lambda x_{kl}^\mu) 
& = \alpha_0\Bigl(\sum_{\nu\prec \lambda \otimes \mu}z_\nu \Bigr)
= \sum_{\nu\prec \lambda \otimes \mu} \alpha(|\nu|) \, \varepsilon(z_\nu) \\
& = \sum_{\nu\prec \lambda \otimes \mu} \alpha(|\lambda||\mu|) \, \varepsilon(z_\nu)
= \alpha(|\lambda|) \, \alpha(|\mu|) \, \varepsilon(x_{ij}^\lambda x_{kl}^\mu) \\
& = \alpha_0(x_{ij}^\lambda) \, \alpha_0(x_{kl}^\mu) \, .
\end{align*}
Hence, $\alpha_0 \in \H_{\ell}^1(H,k)$ with $\alpha_0'= \alpha$.  
This concludes the proof of Theorem~\ref{H1cos}.
\end{proof}

\begin{rems}
(a) We have just proved that the group homomorphism
$\H_{\ell}^1(H,k) \to \widehat{\Gamma_H}$ of~\eqref{iso-coss} is an isomorphism. 
Since $\H_{\ell}^1(H,k)$ is also isomorphic to the group of monoidal automorphisms
of the identity functor of the category of $H$-comodules, it is possible
to prove Theorem~\ref{H1cos} by using~\cite[Prop.~3.9]{gn} or 
\cite[Prop.~1.3.3]{pe}.

(b) Theorem~\ref{H1cos} implies that for any cosemisimple Hopf algebra~$H$
over an algebraically closed field of characteristic zero, 
the first lazy homology Hopf algebra depends only on the fusion semiring 
of the tensor category of $H$-comodules.

(c) Let $H = k^G$ be the Hopf algebra of functions on a finite group~$G$.
By Proposition~\ref{HL1k^G} and Theorem~\ref{H1cos},
$$\H_1^{\ell}(H) \cong k^{Z(G)}\cong k[\Gamma_H] \, .$$
Consequently, $\Gamma_H \cong \widehat{Z(G)} = \Hom(Z(G), k^{\times})$,
a result already proved in~\cite{mu}.
\end{rems}

We end this section by showing how to recover M\"uger's description of the center of a compact group
(here the ground field is the field~$\mathbb C$ of complex numbers).

\begin{corollary}
The center $Z(G)$ of a compact group~$G$ is isomorphic
to the group of unitary characters $\widehat{\mathcal C(G)}= \Hom(\mathcal C(G), \mathbb S^1)$ 
of the chain group~$\mathcal C(G)$ of~$G$, as defined in~\cite{mu}.
\end{corollary}

\begin{proof}
Let $\mathcal R(G)$ be the Hopf algebra of (complex) representative functions on $G$;
it is cosemisimple.
The compact group $G$ is isomorphic to the compact group
$\Hom_{*-\alg}(\mathcal R(G),\mathbb C)$ by the Tannaka-Krein theorem, and 
$$Z(G) \cong \Hom_{*-\alg}(\mathcal R(G),\mathbb C) \cap \H_{\ell}^1(\mathcal R(G), \mathbb C) \, .$$
The chain group of~\cite{mu} is nothing but~$\Gamma_{\mathcal R(G)}$, and
the $*$-morphisms correspond to unitary characters under the isomorphism 
$\H_{\ell}^1(\mathcal R(G), \mathbb C) \cong \widehat{\Gamma_{\mathcal R(G)}}$
of~\eqref{iso-coss}.
This concludes the proof.
\end{proof}

\section{The Second Lazy Homology Hopf Algebra}\label{lazyhom2}

In this section we construct the second lazy homology Hopf algebra 
of a Hopf algebra~$H$. 

\subsection{The Hopf algebra $\H^{\ell}_2(H)$}\label{HL2}

It would be very helpful if we could find a version of 
the differential~$d^{\SW}_3$ of Lemma~\ref{d2-cocomm}
under the form of a  Hopf algebra morphism
$F(H \otimes H \otimes H) \to F(H^{[2]})$.
Unfortunately this may not be possible for an arbitrary Hopf algebra.
Instead we proceed as follows.

Mimicking~\eqref{d3a}, we set
\begin{equation}\label{d3}
d_3(x,y,z) = 
t(\widetilde{y_1\otimes z_1}) \, t(\widetilde{x_1 \otimes y_2 z_2}) \, 
{t}^{-1}(\widetilde{x_2y_3 \otimes z_3}) \, {t}^{-1}(\widetilde{x_3 \otimes y_4}) 
\in F(H^{[2]})
\end{equation}
for all $x$, $y$, $z\in H$.

\begin{lemma}\label{d(x,y,z)} 
For all $x,y, z\in H$, we have the following identities:
\begin{equation}\label{d31}
d_2 \bigl (d_3(x,y,z) \bigr) = \varepsilon(xyz)\, 1 \, ,
\end{equation}
\begin{equation}\label{d32}
\Delta\bigl(d_3(x,y,z)\bigr) =
t(\widetilde{y_1\otimes z_1}) \, t(\widetilde{x_1\otimes y_2z_2}) \,
{t}^{-1}(\widetilde{x_3y_4\otimes z_4}) \, {t}^{-1}(\widetilde{x_4\otimes y_5})
\otimes d_3(x_2,y_3,z_3) \, ,
\end{equation}
\begin{equation}\label{d33}
(\id\otimes d_2) \bigl(\Delta(d_3(x,y,z) \bigr) = d_3(x,y,z) \otimes 1 \, .
\end{equation}
Furthermore, $d_3(x,y,z) \in \HKer(d_2)$ for all $x,y,z \in H$.
\end{lemma} 

\begin{proof}
Identity\,\eqref{d31} is a consequence of the following computation, 
in which the strong cocommutativity of~$H^{[1]}$ is used several times. 
We have
\begin{align*}
d_2\bigl (d_3(x,y,z) \bigr) 
& = d_2(t(\widetilde{y_1\otimes z_1})) \, 
d_2(t(\widetilde{x_1 \otimes y_2 z_2})) \, \\
& \hskip 85pt \times
d_2({t}^{-1}(\widetilde{x_2y_3 \otimes z_3})) \, 
d_2({t}^{-1}(\widetilde{x_3 \otimes y_4})) \\
& = t(\underline{y_1}) \, t(\underline{z_1}) \, {t}^{-1}(\underline{y_2z_2}) \,
t(\underline{x_1}) \, t(\underline{y_3z_3}) \, {t}^{-1}(\underline{x_2y_4z_4})     \\
& \hskip 85pt \times
d_2({t}^{-1}(\widetilde{x_3y_5 \otimes z_5})) \, 
d_2({t}^{-1}(\widetilde{x_4 \otimes y_6})) \\
& = t(\underline{x_1}) \, t(\underline{y_1}) \, 
t(\underline{z_1}) \,  {t}^{-1}(\underline{x_2y_2z_2})  \\
& \hskip 85pt \times
d_2({t}^{-1}(\widetilde{x_3y_3 \otimes z_3})) \, 
d_2({t}^{-1}(\widetilde{x_4 \otimes y_4})) \\
& = t(\underline{x_1}) \, t(\underline{y_1}) \, 
t(\underline{z_1}) \, {t}^{-1}(\underline{x_2y_2z_2}) \,
{t}^{-1}(\underline{x_4y_4}) \,
{t}^{-1}(\underline{z_4}) \, \\
& \hskip 85pt \times
t(\underline{x_3y_3z_3}) \,
d_2({t}^{-1}(\widetilde{x_5 \otimes y_5})) \\
& = t(\underline{x_1}) \, t(\underline{y_1}) \, 
t(\underline{z_1}) \, 
{t}^{-1}(\underline{z_2}) \, {t}^{-1}(\underline{x_2y_2}) \,
d_2({t}^{-1}(\widetilde{x_3 \otimes y_3})) \\
& = \varepsilon (z) \, t(\underline{x_1}) \, t(\underline{y_1}) \, 
{t}^{-1}(\underline{x_2y_2}) \, 
d_2({t}^{-1}(\widetilde{x_3 \otimes y_3})) \\
& = \varepsilon(z) \, d_2(t(\widetilde{x_1 \otimes y_1}) \,
{t}^{-1}(\widetilde{x_2 \otimes y_2})) \\
& = \varepsilon(xyz)\, 1 \, .
\end{align*}
For Identity\,\eqref{d32} , we use the lazy cocommutativity of~$H^{[2]}$. We obtain
\begin{align*}
\Delta \bigl( d_3(x,y,z) \bigr)  & =
t(\widetilde{y_1\otimes z_1}) \, t(\widetilde{x_1\otimes y_3z_3}) \,
{t}^{-1}(\widetilde{x_4y_6 \otimes z_6}) \, {t}^{-1}(\widetilde{x_6\otimes y_8}) \\
& \hskip 45pt \otimes t(\widetilde{y_2\otimes z_2}) \, 
t(\widetilde{x_2\otimes y_4z_4}) \,
{t}^{-1}(\widetilde{x_3y_5 \otimes z_5}) \, {t}^{-1}(\widetilde{x_5\otimes y_7}) \\
& =  t(\widetilde{y_1\otimes z_1}) \, t(\widetilde{x_1\otimes y_2z_2}) \,
{t}^{-1}(\widetilde{x_5y_7\otimes z_6}) \, {t}^{-1}(\widetilde{x_6\otimes y_8}) \\
& \hskip 45pt \otimes t(\widetilde{y_3\otimes z_3}) \,
t(\widetilde{x_2y_4\otimes z_4}) \,
{t}^{-1}(\widetilde{x_3\otimes y_5z_5}) \, {t}^{-1}(\widetilde{x_4\otimes y_6}) \\
& =  t(\widetilde{y_1\otimes z_1}) \, t(\widetilde{x_1\otimes y_2z_2}) \, \\
& \hskip 45pt 
{t}^{-1}(\widetilde{x_3\otimes y_4z_4}) \, {t}^{-1}(\widetilde{x_4\otimes y_5})
\otimes d_3(x_2,y_3,z_3) \, .
\end{align*}
Identity\,\eqref{d33}  follows from\,\eqref{d31} and\,\eqref{d32} . The final assertion follows from Corollary~\ref{HKerd2}.
\end{proof}

It follows from Lemma~\ref{d(x,y,z)} that the ideal~$\B_2^{\ell}(H)$ of 
$\HKer(d_2)$ generated by the elements
$$d_3(x,y,z) - \varepsilon(xyz)1 \quad \text{and} \quad
S \bigl( d_3(x,y,z) \bigr) - \varepsilon(xyz) $$
for all $x,y,z \in H$, is a Hopf ideal in~$\HKer(d_2)$.

\begin{definition} 
The second lazy homology Hopf algebra $\H^{\ell}_2(H)$ of 
a Hopf algebra~$H$ is the quotient Hopf algebra  
\begin{equation*}
\H^{\ell}_2(H)
= \HKer(d_2) / \B^{\ell}_2(H) \, .
\end{equation*}
\end{definition}

Here again we have strongly modeled the definition of~$\H^{\ell}_2(H)$ 
on the simplicial object~$F(\Gamma_*(H))$.

Observe that $\H^{\ell}_2(H)$ is a \emph{commutative} Hopf algebra
since it is the quotient of a Hopf subalgebra of
the commutative Hopf algebra~$F(H^{[2]})$.

When $H$ is cocommutative, we have the following isomorphism
with the Sweedler homology defined in Section~\ref{Sw-homology}.

\begin{proposition}\label{HL2HS2}
For any cocommutative Hopf algebra~$H$ we have
$$\H^{\ell}_2(H) = \H^{\SW}_2(H)  \, .$$
\end{proposition}

\begin{proof}
As observed in the proof of Proposition~\ref{HS1=HL1},
$F(H^{[2]}) = F(H\otimes H)$ and $d_2 = d^{\SW}_2$.
The map~$d_3$ coincides with the differential~$d^{\SW}_3$ of Lemma~\ref{d2-cocomm}
in view of~\eqref{d3} and~\eqref{d3a}. The conclusion follows.
\end{proof}

When $H=k[G]$ is a group algebra, Propositions~\ref{HSkG}
and~\ref{HL2HS2} imply that
\begin{equation*}
\H^{\ell}_2(k[G]) \cong \H^{\SW}_2(k[G]) \cong k[H_2(G,\ZZ)] \, .
\end{equation*}

\subsection{A universal coefficient theorem}\label{UCT}

We fix a commutative algebra~$R$.
Our next aim is to relate the lazy cohomology group $\H_{\ell}^2(H,R)$ 
of Definition~\ref{lazycoh2} to the above-defined Hopf algebra~$\H^{\ell}_2(H)$.

Recall from~\eqref{ZL2} the group~$Z_{\ell}^2(H,R)$ of 
normalized lazy $2$-cocycles with coefficients in~$R$.

\begin{proposition}\label{kappa}
There is a homomorphism
\begin{equation*}
Z_{\ell}^2(H,R)  \to \Alg(\H^{\ell}_2(H),R) \, , \;
\sigma  \mapsto \widehat{\sigma}
\end{equation*}
defined by 
$\widehat{\sigma}([t(\widetilde{x\otimes y})]) = \sigma(x \otimes y)$
for all $x,y \in H$.
It induces a homomorphism 
$$\kappa : \H_{\ell}^2(H,R) \to \Alg(\H^{\ell}_2(H),R)\, .$$
\end{proposition}

\begin{proof}
Let $\sigma \in Z^2_{\ell}(H,R)$. Since $\sigma \in \reg_{\ell}^2(H,R)$, 
Proposition~\ref{H2motivation} furnishes an element 
$\tilde{\underline{\sigma}} \in \Alg(F(H^{[2]}),R)$
such that for all $x,y \in H$,
$$\tilde{\underline{\sigma}} \bigl( t(\widetilde{x\otimes y}) \bigr)
= \sigma(x \otimes y)
\quad\text{and}\quad
\tilde{\underline{\sigma}} \bigl( t^{-1}(\widetilde{x\otimes y}) \bigr)
= \sigma^{-1}(x \otimes y) \, ,$$
where $\sigma^{-1}$ is the convolution inverse of~$\sigma$.
Restricting~$\tilde{\underline{\sigma}}$ to~$\HKer(d_2)$,
we obtain
\begin{align*}
\tilde{\underline{\sigma}} \bigl( d_3(x, y, z) \bigr)
& = \sigma(y_1 \otimes z_1) \, \sigma(x_1 \otimes y_2z_2) \, 
\sigma^{-1}(x_2y_3 \otimes z_3) \, \sigma^{-1}(x_3 \otimes y_4) \\
& = \sigma(x_1 \otimes y_1) \, \sigma(x_2y_2 \otimes z_1) \, 
\sigma^{-1}(x_3y_3 \otimes z_2) \, \sigma^{-1}(x_4 \otimes y_4) \\
& = \sigma(x_1 \otimes y_1) \, \varepsilon(x_2y_2z) \, \sigma^{-1}(x_3 \otimes y_3) \\
& = \varepsilon(z) \, \sigma(x_1 \otimes y_1) \, \sigma^{-1}(x_2 \otimes y_2) \\
& = \varepsilon(xyz) \, ,
\end{align*}
for all $x,y,z \in H$.
The second equality above is a consequence of~\eqref{2cocycle}.
A~similar computation using the fact that $\sigma$ satisfies~\eqref{lazy2} 
and~\eqref{2cocycle} shows that
$$\tilde{\underline{\sigma}}\bigl(S(d_3(x, y, z))\bigr) = \varepsilon(xyz)$$
for all $x,y,z \in H$.
Therefore, $\tilde{\underline{\sigma}}$ vanishes on the ideal~$\B^{\ell}_2(H)$;
hence, it induces the desired algebra morphism, 
which we denote by~$\widehat{\sigma}$.
One checks that the map $\sigma \mapsto \widehat{\sigma}$ defines a homomorphism
$$Z^2_{\ell}(H,R) \rightarrow \Alg(\H^{\ell}_2(H),R) \, .$$

Let $\mu \in \Reg^1_{\ell}(H,R)$, so that $\partial(\mu) \in Z^2_{\ell}(H,R)$. 
By Propositions~\ref{Alg-Reg} and~\ref{C1motivation}, $\mu$~induces the algebra morphism 
$\tilde{\underline{\mu}} : F(H^{[1]}) \to R$ given by
$\tilde{\underline{\mu}}(t(\underline{x})) = \mu(x)$ 
and $\tilde{\underline{\mu}}(t^{-1}(\underline{x})) = \mu^{-1}(x)$ for all $x\in H$.
We claim that
$$\widetilde{\underline{\partial(\mu)}} = \tilde{\underline{\mu}} \circ d_2 \, .$$
Indeed, by~\eqref{coboundary} and~\eqref{d2c}, for all $x,y \in H$ we obtain
\begin{align*}
\widetilde{\underline{\partial(\mu)}}\bigl( t(\widetilde{x \otimes y} )\bigr)
& = \partial(\mu)(x\otimes y)
= \mu(x_1) \, \mu(y_1) \, \mu^{-1}(x_2y_2) \\
& = \tilde{\underline{\mu}}(t(\underline{x_1})) \, \tilde{\underline{\mu}}(t(\underline{y_1})) \, 
\tilde{\underline{\mu}}({t}^{-1}(\underline{x_2y_2})) \\
& = \tilde{\underline{\mu}} \bigl( d_2(t(\widetilde{x \otimes y})) \bigr) \, .
\end{align*}
By a similar computation,
$\widetilde{\underline{\partial(\mu)}}\bigl( t^{-1}(\widetilde{x \otimes y} )\bigr)
= \tilde{\underline{\mu}} \bigl( d_2(t^{-1}(\widetilde{x \otimes y})) \bigr)$
for all $x,y \in H$. This proves the claim.

If $a \in \HKer(d_2)$, then $d_2(a) =\varepsilon(a)\, 1$, 
so that $\widetilde{\underline{\partial(\mu)}} = \varepsilon$;
hence, we obtain the desired homomorphism
$\H_{\ell}^2(H,R) \to \Alg(\H^{\ell}_2(H),R)$.
\end{proof}

For any $f\in \Alg(\Im(d_2),R)$, we set
\begin{equation}\label{def-delta}
(\delta f)(x \otimes y) = f \Bigl(d_2\bigl(t(\widetilde{x \otimes y}) \, t^{-1}(\widetilde{1 \otimes 1}) \bigr)\Bigr)
= f \bigl(t(\underline{x_1}) \, 
t(\underline{y_1}) \, t^{-1}(\underline{x_2y_2}) \, t^{-1}(\underline{1})\bigr)
\end{equation}
for all $x, y \in H$.
This defines an element of~$\Hom(H\otimes H,R)$.

\begin{lemma}\label{delta}
(a) For any $f\in \Alg(\Im(d_2),R)$, the map $\delta f \in \Hom(H\otimes H,R)$ 
is a lazy $2$-cocycle.

(b) If $f, f'\in \Alg(\Im(d_2),R)$, then
$\delta (f *f') = (\delta f)* (\delta f')$.

(c) For any $g\in \Alg(F(H^{[1]}),R)$, we have
$$\delta (\iota^*(g)) = \partial (\check{g})\, ,$$
where $\check{g} \in \Reg^1_{\ell}(H,R)$ is defined by
$\check{g}(x) = g(t(\underline{x})\, t^{-1}(\underline{1}))$ ($x\in H$).
\end{lemma}

\pf
(a) The map $\delta f$ is invertible with respect to the convolution product.
Indeed, it is easy to check that the map
$$x \otimes y \mapsto f \bigl(
t(\underline{x_1y_1}) \, t^{-1}(\underline{x_2}) \, 
t^{-1}(\underline{y_2}) \, t(\underline{1}) \bigr)$$
is an inverse for~$\delta f$.

Let us verify that $\delta f$ is normalized. Indeed, for $x\in H$,
$$(\delta f)(x \otimes 1) 
= f \bigl(t(\underline{x_1}) \, t^{-1}(\underline{x_2}) \, t(\underline{1}) \, t^{-1}(\underline{1})\bigr)
= \eps(x) f(1) = \eps(x)\, 1_R\, .$$
Similarly, $(\delta f)(1 \otimes x) = \eps(x)\, 1_R$.

By Proposition~\ref{C1coalgebra}, 
$t^{\pm 1}(\underline{x_1}) \otimes x_2 = t^{\pm 1}(\underline{x_2}) \otimes x_1$
for all $x\in H$. Therefore,
\begin{eqnarray*}
(\delta f)(x_1 \otimes y_1) \otimes x_2 y_2
& = &
f \bigl(t(\underline{x_1}) \, 
t(\underline{y_1}) \, t^{-1}(\underline{x_2y_2}) \, t^{-1}(\underline{1}) \bigr) \otimes x_3 y_3\\
& = & 
f \bigl(t(\underline{x_1}) \, 
t(\underline{y_1}) \, t^{-1}(\underline{x_3y_3}) \, t^{-1}(\underline{1}) \bigr) \otimes x_2 y_2\\
& = & 
f \bigl(t(\underline{x_2}) \, 
t(\underline{y_2}) \, t^{-1}(\underline{x_3y_3}) \, t^{-1}(\underline{1}) \bigr) \otimes x_1 y_1\\
& = &
(\delta f)(x_2 \otimes y_2) \otimes x_1 y_1 \, .
\end{eqnarray*}
This shows that $\delta f$ is lazy.

Let us finally check that $\delta f$ is a left $2$-cocycle.
On one hand, for all $x,y,z \in H$,
\begin{eqnarray*}
(\delta f) (x_1,y_1) \, \delta f (x_2y_2,z)
& = & f \bigl(t(\underline{x_1}) \, 
t(\underline{y_1}) \, t^{-1}(\underline{x_2y_2}) \, t^{-1}(\underline{1}) \bigr) \\
&& \hskip 45pt \times
f \bigl(t(\underline{x_3y_3}) \, 
t(\underline{z_1}) \, t^{-1}(\underline{x_4y_4z_2}) \, t^{-1}(\underline{1}) \bigr) \\
& = & f \bigl(t(\underline{x_1}) \, 
t(\underline{y_1}) \,  t^{-1}(\underline{x_2y_2}) \, 
t(\underline{x_3y_3}) \\
&& \hskip 45pt \times
t(\underline{z_1}) \, t^{-1}(\underline{x_4y_4z_2}) \, t^{-1}(\underline{1})^2 \bigr) \\
& = & f \bigl(t(\underline{x_1}) \, t(\underline{y_1}) \, 
t(\underline{z_1}) \, t^{-1}(\underline{x_2y_2z_2}) \, t^{-1}(\underline{1})^2 \bigr) \, .
\end{eqnarray*}
On the other hand,
\begin{eqnarray*}
(\delta f) (y_1,z_1) \, \delta f (x,y_2z_2)
& = & 
f\bigl(t(\underline{y_1}) \, 
t(\underline{z_1}) \, t^{-1}(\underline{y_2z_2}) \, t^{-1}(\underline{1}) \bigr)\\
&& \hskip 45pt \times
f \bigl(t(\underline{x_1}) \, 
t(\underline{y_3z_3}) \, t^{-1}(\underline{x_2y_4z_4}) \, t^{-1}(\underline{1}) \bigr) \\
& = & 
f\bigl(t(\underline{x_1}) \, t(\underline{y_1}) \, t(\underline{z_1}) \\
&& \hskip 45pt \times
t^{-1}(\underline{y_2z_2})  \, 
t(\underline{y_3z_3}) \, t^{-1}(\underline{x_2y_4z_4}) \, t^{-1}(\underline{1})^2 \bigr) \\
& = & 
f\bigl(t(\underline{x_1}) \, t(\underline{y_1}) \, 
t(\underline{z_1}) \, t^{-1}(\underline{x_2y_2z_2}) \, t^{-1}(\underline{1})^2 \bigr) \, ,
\end{eqnarray*}
which by the above computation is equal to
$(\delta f) (x_1,y_1) \, (\delta f) (x_2y_2,z)$.

(b) Since $d_2$ is a coalgebra morphism,
\begin{eqnarray*}
\delta (f *f')(x\otimes y)
& = & (ff') \bigl(d_2(t(\widetilde{x \otimes y}) \, t^{-1}(\widetilde{1 \otimes 1}))\bigr) \\
& = & f \bigl(d_2(t(\widetilde{x \otimes y})\, t^{-1}(\widetilde{1 \otimes 1}))_1\bigr)\\
&& \hskip 45pt \times
f' (\bigl(d_2(t(\widetilde{x \otimes y})\, t^{-1}(\widetilde{1 \otimes 1}))_2\bigr) \\
& = & f \bigl(d_2(t(\widetilde{x_1 \otimes y_1}) \, t^{-1}(\widetilde{1 \otimes 1}))\bigr)\\
&& \hskip 45pt \times
f' (\bigl(d_2(t(\widetilde{x_2 \otimes y_2})\, t^{-1}(\widetilde{1 \otimes 1}))\bigr) \\
& = & (\delta f) * (\delta f')(x\otimes y) \, .
\end{eqnarray*}

(c) Since $g$ is an algebra morphism on~$F(H^{[1]})$,
\begin{eqnarray*}
(\delta (\iota^*g))(x\otimes y)
& = & 
g \bigl(t(\underline{x_1}) \, 
t(\underline{y_1}) \, t^{-1}(\underline{x_2y_2}) \, t^{-1}(\underline{1}) \bigr) \\
& = & 
g \bigl(t(\underline{x_1}) \, t^{-1}(\underline{1}) \bigr)\, 
g\bigl( t(\underline{y_1}) \, t^{-1}(\underline{1}) \bigr)\, 
g\bigl( t^{-1}(\underline{x_2y_2}) \, t(\underline{1}) \bigr) \\
& = &
\check{g}(x_1) \, \check{g}(y_1) \, \check{g}^{-1}(x_2y_2) \\
& = &  \partial (\check{g})(x\otimes y)\, .
\end{eqnarray*}
\epf

We now relate the lazy cohomology group~$\H^2_{\ell}(H,R)$
to the group $\Ext^1(H,R)$ introduced in Section~\ref{ext}.

\begin{proposition}\label{ext-h2}
The map $f\in \Alg(\Im(d_2),R) \mapsto \delta f \in \Hom(H\otimes H, R)$
induces an injective group homomorphism
$$\delta_{\#} : \Ext^1(H,R) \to \H^2_{\ell}(H,R)\, .$$
\end{proposition}

\pf
It follows from the definitions and Lemma~\ref{delta}
that $\delta$ induces the desired group homomorphism.
Let us check that $\delta_{\#}$ is injective. 
Let $f\in \Alg(\Im(d_2),R)$ be such that 
$\delta f = \partial(\mu)$ for some $\mu \in \Reg^1_{\ell}(H,R)$.
This means that for all $x,y \in H$,
$$(\delta f)(x\otimes y) = \mu(x_1) \, \mu(y_1) \, \mu^{-1}(x_2y_2)
= \mu(x_1) \, \mu(y_1) \, \mu^{-1}(x_2y_2) \, \mu^{-1}(1)$$
since $\mu$ is normalized.
By Propositions~\ref{Alg-Reg} and~\ref{C1motivation} there is a unique 
morphism $g\in \Alg(F(H^{[1]}),R)$ such that
$g(t(\underline{x})) = \mu(x)$ and $g(t^{-1}(\underline{x})) = \mu^{-1}(x)$
for all~$x\in H$.
Restricting~$g$ to~$\Im(d_2)$, we obtain
\begin{eqnarray*}
g\bigl(d_2(t(\widetilde{x \otimes y}) \, t^{-1}(\widetilde{1 \otimes 1}))\bigr) & = &
g\bigl( t(\underline{x_1}) \, t(\underline{y_1}) \, t^{-1}(\underline{x_2y_2}) \, t^{-1}(\underline{1}) \bigr) \\
& = & g(t(\underline{x_1})) \, g(t(\underline{x_2})) \, g(t^{-1}(\underline{x_2y_2}))  \, g(t^{-1}(\underline{1}))\\
& = & \mu(x_1) \, \mu(y_1) \, \mu^{-1}(x_2y_2) \, \mu^{-1}(1) \\
& = & (\delta f)(x\otimes y)  \\
& = & f\bigl(d_2(t(\widetilde{x \otimes y})\, t^{-1}(\widetilde{1 \otimes 1}))\bigr)\, .
\end{eqnarray*}
Since $d_2(t^{-1}(\widetilde{1 \otimes 1}))$ is invertible 
(with inverse $d_2(t(\widetilde{1 \otimes 1}))$), for all $x,y \in H$, we obtain
$$g\bigl(d_2(t(\widetilde{x \otimes y}))\bigr) = 
f\bigl(d_2(t(\widetilde{x \otimes y}))\bigr)\, .$$
This shows that $f$ is in the image of 
$\iota^* : \Alg(F(H^{[1]}),R) \to \Alg(\Im(d_2),R)$,
hence is zero in~$\Ext^1(H,R)$.
\epf

Now we combine the group homomorphisms of Propositions~\ref{kappa} and~\ref{ext-h2}:
\begin{equation*}
\Ext^1(H,R) \overset{\delta_{\#}}{\longrightarrow} \H^2_{\ell}(H,R)
\overset{\kappa}{\longrightarrow} \Alg(\H_2^{\ell}(H),R) \, .
\end{equation*}
We already know that the left arrow $\delta_{\#}$ is injective.

We can now state a universal coefficient theorem for~$\H^2_{\ell}(H,R)$.

\begin{theorem}\label{UCT2}
For any Hopf algebra~$H$ and any commutative algebra~$R$,
the sequence of groups
\begin{equation*}
1 \longrightarrow \Ext^1(H,R) \overset{\delta_{\#}}{\longrightarrow} \H^2_{\ell}(H,R)
\overset{\kappa}{\longrightarrow} \Alg(\H_2^{\ell}(H),R) 
\end{equation*}
is exact.
If in addition the ground field $k$ is algebraically closed, then
\begin{equation*}
\H^2_{\ell}(H,k) \cong \Alg(\H_2^{\ell}(H),k) \, .
\end{equation*}
\end{theorem}

\pf
Since $\delta_{\#}$ is injective, it suffices to prove that
$\Ker(\kappa) = \Im(\delta_{\#})$.
To this end, consider the sequence of Hopf algebra morphisms
\begin{equation}\label{exactd2}
k \longrightarrow \HKer(d_2) \overset{j}{\longrightarrow} F(H^{[2]})
\overset{d_2}{\longrightarrow} \Im(d_2) \longrightarrow k\, ,
\end{equation}
where $j$ is the inclusion of $\HKer(d_2)$ into~$F(H^{[2]})$.
We claim that this sequence is exact. 
Indeed, $j$ is injective, $d_2$ is surjective, and
\begin{equation*}
\Ker(d_2) = \HKer(d_2)^+ F(H^{[2]})\, .
\end{equation*}
The latter equality is a consequence of~Corollary~\ref{HKerd2} and
of~\cite[Th.~4.3]{ta72} (see also the latter's proof).
By Proposition~\ref{Alg-exact}, the exact sequence~\eqref{exactd2} 
induces the exact sequence of groups
\begin{equation}\label{exact2}
1 \longrightarrow \Alg(\Im(d_2),R) \overset{d_2^*}{\longrightarrow} \Alg(F(H^{[2]}),R)
\overset{j^*}{\longrightarrow} \Alg(\HKer(d_2),R) \, .
\end{equation}
Now consider an element of~$\H^2_{\ell}(H,R)$ and represent it by
an element $\sigma \in Z^2_{\ell}(H,R) \subset \reg_{\ell}^2(H,R)$. 
If its image under~$\kappa$ is trivial, then the corresponding algebra morphism 
$\tilde{\underline{\sigma}} \in \Alg(F(H^{[2]}),R)$ given by
$$\tilde{\underline{\sigma}} \bigl( t(\widetilde{x\otimes y}) \bigr)
= \sigma(x \otimes y)$$
restricts to the trivial element of~$\Alg(\HKer(d_2),R)$.
By exactness of~\eqref{exact2}, there is $f\in \Alg(\Im(d_2),R)$ such that
$\delta f = \tilde{\underline{\sigma}}$.
It follows that the element of~$\H^2_{\ell}(H,R)$ represented by~$\sigma$
is in the image of the homomorphism~$\delta_{\#}$ defined in Proposition~\ref{ext-h2}.

Assume now that $k$ is algebraically closed. 
It follows from the exactness of~\eqref{exactd2} and from Proposition~\ref{Alg-surj}
that the map 
$$j^* : \Alg(F(H^{[2]}),k) \to \Alg(\HKer(d_2),k)$$
of~\eqref{exact2} is surjective. 
Since by Lemma~\ref{inject},
$\Alg(\H_2^{\ell}(H),k)$ embeds into $\Alg(\HKer(d_2),k)$,
the map $\kappa:  \H^2_{\ell}(H,R) \to \Alg(\H_2^{\ell}(H),R)$ is surjective.
The injectivity of~$\kappa$ follows from the first part of the theorem
and the vanishing of~$\Ext^1(H,k)$, which is a consequence of Proposition~\ref{Ext-zero}. 
\epf

\section{Computations for the Sweedler Algebra}\label{Swalgebra}

In this section $k$ will denote a field of characteristic~$\neq 2$.
We compute the lazy homology of Sweedler's four-dimensional Hopf algebra.

Recall that the Sweedler algebra is defined by the following presentation: 
$$H_4 = k\, \langle \, x,g \ | \ g^2=1, \ x^2=0, \ xg = - gx \, \rangle \, .$$
The algebra $H_4$ has $\{1, g, x, y = xg\}$ as a linear basis.
It is a Hopf algebra with coproduct~$\Delta$, counit~$\varepsilon$,
and antipode~$S$ given by
$$\Delta(g) = g \otimes g \, ,  \quad \Delta(x) = 1 \otimes x + x \otimes g\, ,$$
$$\varepsilon(g) = 1\, , \quad \varepsilon(x)=0\, , 
\quad S(g)=g \, , \ S(x) = - y \, .$$

\begin{theorem}\label{Swthm}
The lazy homology Hopf algebras of the Sweedler algebra $H_4$ are given by
$$\H_1^\ell(H_4) \cong k \quad\text{and}\quad 
\H_2^\ell(H_4) \cong k[X] \, ,$$ 
where $X$ is a primitive element.
\end{theorem}

The rest of the section is devoted to the proof of this theorem.
We start by computing the coalgebra~$H_4^{[1]}$, as defined in Section~\ref{lazy-quot1}.

\begin{lemma}\label{Swlem1}
The coalgebra $H_4^{[1]}$ is one-dimensional, spanned by the grouplike element~$\underline{1}$. 
\end{lemma}

\begin{proof} An easy computation shows that $\underline{1} = \underline{g}$ and 
$\underline{x} = \underline{y} = 0$ in~$H_4^{[1]}$. Since $1$ is grouplike, so is~$\underline{1}$.
\end{proof}

It follows from Lemma~\ref{Swlem1} and the definition of~$F(H_4^{[1]})$ that there is a
Hopf algebra isomorphism 
$$f : k[T,T^{-1}] \to F(H_4^{[1]})$$
determined by $f(T^{\pm 1}) =  t^{\pm 1}(\underline{1})$.
By our first definition of~$\H_1^\ell$ (see Section~\ref{HL1a}), the Hopf algebra
$\H_1^\ell(H_4)$ is isomorphic to the quotient of $k[T,T^{-1}]$ by the ideal generated
by $T^2 - T$, which in view of the invertibility of~$T$ is the same as the ideal generated
by $T-1$. We thus obtain the desired isomorphism $\H_1^\ell(H_4) \cong k$.

Let us next determine the Hopf algebra~~$H_4^{[2]}$, as defined in Section~\ref{lazy-quot2}.

\begin{lemma}\label{Swlem2} 
The coalgebra $H_4^{[2]}$ is five-dimensional with basis 
$\{h_0, h_1, h_2$, $h_3, h_4\}$, where
$$h_0 = \widetilde{1 \otimes 1} \, , \;\; h_1 =  \widetilde{x \otimes x} \, , \;\;
h_2 = \widetilde{x \otimes y} \, , \;\;  h_3 = \widetilde{y \otimes x} \, , \;\; 
h_4 = \widetilde{y \otimes y} \, ,$$
and with coproduct
$$\Delta(h_0)  = h_0 \otimes h_0 \quad \text{and} \quad
\Delta(h_i) = h_0 \otimes h_i + h_i \otimes h$$
for $i = 1, 2,3,4$.
\end{lemma}

\begin{proof} 
By a direct computation one shows that $H_4^{[2]}$ is obtained from $H_4 \otimes H_4$
by killing the eight elements
$\widetilde{1 \otimes x}$, $\widetilde{1 \otimes y}$, $\widetilde{g \otimes x}$, $\widetilde{g \otimes y}$,
$\widetilde{x \otimes 1}$, $\widetilde{y \otimes 1}$, $\widetilde{x \otimes g}$, $\widetilde{y \otimes g}$,
and adding the relations
$$\widetilde{g \otimes 1} = \widetilde{1 \otimes g} = \widetilde{g \otimes g} = h_0 \, .$$
It follows that $H_4^{[2]}$ is spanned by $\{h_0, h_1, h_2, h_3, h_4\}$.
The formulas for the coproduct follow easily.
\end{proof}

Let $B= k[T,T^{-1}, Y_1,Y_2,Y_3,Y_4]$ be the commutative Hopf algebra
with coproduct determined by
$$\Delta(T) = T \otimes T  \quad \text{and} \quad
\Delta(Y_i) = T \otimes Y_i + Y_i \otimes T$$
for $i = 1, 2,3,4$. Then it follows from Lemma~\ref{Swlem2} 
that there is a Hopf algebra isomorphism $g : B \to F(H_4^{[2]})$ such that
$g(T) = h_0$ and $g(Y_i) = h_i$ for $i = 1, 2,3,4$.

\begin{lemma}\label{Swlem3}
 Let $A= k[X_1,X_2,X_3,X_4]$ be the commutative Hopf algebra
such that $X_1, X_2, X_3, X_4$ are primitive. Then
there is a commutative diagram of Hopf algebras
 $$\xymatrix{
 k \ar[r] & A \ar[r]^{\nu} \ar[d] & B\ar[r]^{\pi} \ar[d]^g
 & k[T, T^{-1}] \ar[r] \ar[d]^f & k \\
 k \ar[r] & {\rm HKer}(d_2)  \ar[r] & F(H_4^{[2]}) \ar[r]^{d_2} & F(H_4^{[1]}) \ar[r] & k}
 $$
 in which the horizontal sequences are exact and the vertical morphisms are isomorphisms.
\end{lemma}

\begin{proof}
The Hopf algebra morphism $\nu : A \to B$ is defined by
 $\nu(X_i) = T^{-1}Y_i$ for $i= 1,2,3,4$; it is clearly injective. 
The Hopf algebra morphism
 $\pi : B \to k[T,T^{-1}]$ sends $T$ to itself and $Y_i$ to~$0$ ($1 \leq i \leq 4$); 
it is surjective.
It is clear that $\nu(A)^+B={\rm Ker}(\pi)$, which implies that
the top sequence is exact with ${\HKer}(\pi) = \nu(A)$. 
The Hopf algebra morphism $d_2 : F(H_4^{[2]}) \to F(H_4^{[1]})$
sends $t(\widetilde{1 \otimes 1})$ to~$t(\overline{1})$ and the remaining
generators to~$0$. Therefore $d_2$ is surjective and the right square is commutative.
The left vertical map is constructed by sending 
each generator $X_i$ to the appropriate element in ${\HKer}(d_2)$;
for instance, $X_1$ is sent to $t^{-1}(\widetilde{1 \otimes 1})t(\widetilde{x \otimes x})$).
The conclusion follows.  
\end{proof}

By Lemma~\ref{Swlem3} the map $d_2$ is surjective. It thus follows from Proposition~\ref{HS1=HL1}
that
$\H^{\ell}_1(H_4)  \cong F(H_4^{[1]}) \quot \Im(d_2) \cong k$,
which gives another proof of the first isomorphism in Theorem~\ref{Swthm}.

Let us now compute the values of the trilinear map 
$$d_3 : H_4 \times H_4 \times H_4 \to F(H_4^{[2]}) \, .$$

\begin{lemma}\label{Swlem4}
For all $a,b \in H_4$,
\begin{equation}\label{Swlem41}
d_3(a,b,1) = d_3(a,1,b) = d_3(1,a,b) = \varepsilon(ab) \, ,
\end{equation}
\begin{equation}\label{Swlem42}
d_3(a,g,g) = d_3(g,a,g) = d_3(g,g,a) = \varepsilon(a) \, .
\end{equation}
If we set
$$a_1 = t^{-1}(\widetilde{1 \otimes 1})t(\widetilde{x \otimes x})\, , \quad
a_2=t^{-1}(\widetilde{1 \otimes 1})t(\widetilde{x \otimes y}) \, ,$$
$$a_3 = t^{-1}(\widetilde{1 \otimes 1})t(\widetilde{y \otimes x})\ , \quad
a_4=t^{-1}(\widetilde{1 \otimes 1})t(\widetilde{y \otimes y}) \, ,$$
then
$\eps(a_1) = \eps(a_2) = \eps(a_3) = \eps(a_4) = 0$ and
$$d_3(x,x,g) = -a_1+a_2 = - d_3(x,y,g) \, ,$$
$$d_3(y,x,g) = -a_3+a_4 = - d_3(y,y,g) \, ,$$
$$d_3(y,g,x) = -a_1-a_4 =  d_3(x,g,y) \, ,$$
$$d_3(x,g,x) = -a_2-a_3 =  d_3(y,g,y) \, ,$$
$$d_3(g,x,x) = a_1+a_3 = d_3(g,y,x) \, ,$$
$$d_3(g,x,y) = a_2+a_4 = d_3(g,y,y) \, ,$$
$$d_3(x,x,x) = d_3(x,y,x)= d_3(y,x,x)=d_3(y,y,x)=0 \, ,$$
$$d_3(x,x,y) = d_3(y,x,y)=d_3(x,y,y)=d_3(y,y,y)=0 \, .$$
\end{lemma}

\begin{proof}
It follows from Lemma~\ref{Swlem2} that
$$t(\widetilde{a \otimes 1}) = t(\widetilde{1 \otimes a}) = \varepsilon(a)t(\widetilde{1 \otimes 1})
=t(\widetilde{a \otimes g})=t(\widetilde{g \otimes a}) \, .$$
From this we easily deduce~\eqref{Swlem41} and~\eqref{Swlem42}.
The remaining identities are obtained by a brute force computation. For instance,
\begin{align*}
d_3(x,x,g) 
& = t(\widetilde{1 \otimes g})t^{-1}(\widetilde{x \otimes x})
+ t(\widetilde{x \otimes y})t^{-1}(\widetilde{g \otimes g}) \\
& = S(a_1) +a_2 = -a_1+a_2 \, .
\end{align*}
The other computations are similar.
\end{proof}

\begin{proof}[Proof of Theorem~\ref{Swthm}]
By Lemma~\ref{Swlem3} and Lemma~\ref{Swlem4}, the Hopf algebra
$\H_2^\ell(H_4)$ is isomorphic to the quotient of $A = k[X_1,X_2,X_3,X_4]$ by the ideal
generated by the relations $X_2=-X_3=-X_4=X_1$.
It follows that $\H_2^\ell(H_4)$ is isomorphic to the polynomial algebra~$k[X]$. 
\end{proof}

Using the universal coefficient theorems \ref{UCT1} and~\ref{UCT2},
we recover the computation of the lazy cohomology of $H_4$ performed in~\cite[Sect.~2]{bc}.

\section*{Acknowledgements}

The present joint work is part of the project ANR BLAN07-3$_-$183390 
``Groupes quantiques~: techniques galoisiennes et d'int\'egration" funded
by Agence Nationale de la Recherche, France.

\end{document}